\documentclass{amsart}
\usepackage{dsfont}
\usepackage[usenames,dvipsnames,svgnames,table]{xcolor}
\usepackage[utf8]{inputenc}
\usepackage[T1]{fontenc}
\usepackage{ tipa }
\usepackage{helvet}
\usepackage{color}
\usepackage{mathrsfs}  
\usepackage{stmaryrd}  
\usepackage{amsmath,amsxtra,amsthm,amssymb,xr,enumerate,enumitem,fullpage,graphicx,mathtools}
\usepackage{fdsymbol}
\usepackage{bold-extra}

\usepackage[symbol]{footmisc}

\usepackage[all]{xy}
\usepackage[bbgreekl]{mathbbol}
\usepackage{bbm}
\usepackage{enumitem}

\DeclareMathSymbol{\invques}{\mathord}{operators}{`>}
\DeclareUnicodeCharacter{00BF}{\tmquestiondown}
\DeclareRobustCommand{\tmquestiondown}{%
  \ifmmode\invques\else\textquestiondown\fi
}
\usepackage{verbatim}
\numberwithin{equation}{section}

\makeatletter
\newcommand{\mylabel}[2]{#2\def\@currentlabel{#2}\label{#1}}
\makeatother

\newtheorem{theorem}{Theorem}[section]
\newtheorem{lemma}[theorem]{Lemma}
\newtheorem{conj}[theorem]{Conjecture}
\newtheorem{proposition}[theorem]{Proposition}
\newtheorem{corollary}[theorem]{Corollary}
\newtheorem{defn}[theorem]{Definition}
\newtheorem{example}[theorem]{Example}

\newtheorem{remark}[theorem]{Remark}

\newcommand{\de}{\mathrm{def}}
\newcommand{\inde}{\mathrm{indef}}
\newcommand{\fR}{\mathfrak{R}}

\newcommand{\fu}{\mathfrak{u}}
\newcommand{\fv}{\mathfrak{v}}

\newcommand{\Thn}{T_{h,n}}

\newcommand{\JL}{\mathrm{JL}}

\newcommand{\Tr}{\operatorname{Tr}}
\newcommand{\Gal}{\operatorname{Gal}}

\newcommand{\CC}{\mathbb{C}}

\newcommand{\QQ}{\mathbb{Q}}
\newcommand{\Qp}{\mathbb{Q}_p}
\newcommand{\Zp}{\mathbb{Z}_p}
\newcommand{\ZZ}{\mathbb{Z}}

\newcommand{\fin}{\f}
\newcommand{\sing}{\mathrm{sing}}

\newcommand{\Tfn}{T_{f,n}}

\newcommand{\Kml}{K_{m,\ell}}

\newcommand{\FF}{\mathbb{F}}
\newcommand{\FFF}{\mathcal{F}}

\newcommand{\ord}{\mathrm{ord}}

\newcommand{\one}{\mathbbm{1}}

\newcommand{\fp}{\mathfrak{p}}

\newcommand{\cH}{\mathcal{H}}
\newcommand{\cO}{\mathcal{O}}

\newcommand{\HIw}{H^1_{\mathrm{Iw}}}

\newcommand{\GL}{\mathrm{GL}}

\newcommand{\col}{\mathrm{Col}}

\newcommand{\fL}{\mathfrak{L}}

\newcommand{\Sel}{\mathrm{Sel}}
\newcommand{\Char}{\mathrm{char}}
\newcommand{\Ind}{\mathrm{Ind}}
\newcommand{\ac}{\textup{ac}}
\newcommand{\Afn}{A_{f,n}}
\newcommand{\LL}{\Lambda}
\newcommand{\TT}{\mathbb{T}}

\newcommand{\f}{\textup{\bf f}}

\newcommand{\lra}{\longrightarrow}
\newcommand{\ra}{\lra}
\newcommand{\res}{\textup{res}}

\newcommand{\ur}{\textup{ur}}
\newcommand{\fP}{\mathfrak{P}}

\newcommand{\cF}{\mathcal{F}}

\newcommand{\fM}{\m}

\newcommand{\cor}{\mathrm{cor}}

\definecolor{Green}{rgb}{0.0, 0.5, 0.0}

\newcommand{\p}{\mathfrak{p}}

\newcommand{\m}{\mathfrak{m}}

\newcommand{\cG}{\mathcal{G}}

\newcommand{\bd}{\mathbb{d}}

\newcommand{\fN}{\mathfrak{N}}

\newcommand{\rank}{\mathrm{rank}}

\newcommand{\Fp}{\mathbb{F}_p}

\newcommand{\cA}{\mathcal{A}}

\newcommand{\binsf}{\bullet\in\{\sharp,\flat\}}

\usepackage{hyperref}

 \definecolor{pAlgae}{RGB}{87,115,135}
\definecolor{airforceblue}{rgb}{0.36, 0.54, 0.66}
	\definecolor{bondiblue}{rgb}{0.0, 0.58, 0.71}
\definecolor{britishracinggreen}{rgb}{0.0, 0.26, 0.15}
\definecolor{camouflagegreen}{rgb}{0.47, 0.53, 0.42}
\definecolor{darkcyan}{rgb}{0.0, 0.55, 0.55}

\hypersetup{
    colorlinks = true,
    linkcolor=blue,
     filecolor=blue,
     citecolor = darkcyan,      
     urlcolor=cyan,
    linkbordercolor = {white},
}

\begin{document}
\author{Ashay Burungale}
\address{Ashay Burungale\newline  Department of Mathematics\\
The University of Texas at Austin\\2515 Speedway\\ Austin\\ TX 78712\\ USA }
\email{ashayk@utexas.edu}

\author{K\^az\i m B\"uy\"ukboduk}
\address{K\^az\i m B\"uy\"ukboduk\newline UCD School of Mathematics and Statistics\\ University College Dublin\\ Belfield\\Dublin 4\\Ireland}
\email{kazim.buyukboduk@ucd.ie}

\author{Antonio Lei}
\address{Antonio Lei\newline Department of Mathematics and Statistics\\University of Ottawa\\
150 Louis-Pasteur Pvt\\
Ottawa, ON\\
Canada K1N 6N5}
\email{antonio.lei@uottawa.ca}

\title{Anticyclotomic Iwasawa theory of abelian varieties of $\mathrm{GL}_2$-type at non-ordinary primes II}

\begin{abstract}
Let $E/\mathbb{Q}$ an elliptic curve with good supersingular reduction at a prime $p\geq 5$, and $K$ an imaginary quadratic field such that the root number of $E$ over $K$ equals $-1$. When $p$ splits in $K$, Castella and Wan formulated the plus/minus Heegner point main conjectures for $E$ along the anticyclotomic $\mathbb{Z}_p$-extension of $K$, and proved them for semistable curves. We generalize their results to two settings:
\begin{itemize}
    \item[1.] For $p$ split in $K$,
    we formulate Sprung-type main conjectures for $\mathrm{GL}_2$-type abelian varieties at non-ordinary primes and prove them under some conditions.
    \item[2.] For $p$ inert in $K$, we formulate, relying on the work of the first-named author with Kobayashi and Ota, plus/minus Heegner point main conjectures for elliptic curves, and prove the minus main conjecture for semistable curves. 
\end{itemize}
The latter yields a $p$-converse to the Gross--Zagier and Kolyvagin theorem for semistable elliptic curves $E$ at supersingular primes $p\geq 5$, 
complementing the pioneering $p$-converse theorems of Skinner and Zhang.

Our method relies on Howard's framework of bipartite Euler systems, Zhang's resolution of Kolyvagin's conjecture and the recent proof of cyclotomic main conjecture at non-ordinary primes. 
\end{abstract}

\maketitle
\tableofcontents

\section{Introduction} 

Iwasawa theory of an elliptic curve $E{/\mathbb{Q}}$ along the anticyclotomic $\mathbb{Z}_p$-extension of an
imaginary quadratic field $K$
concerns the arithmetic of the anticyclotomic $\ZZ_p$-deformation of the $p$-adic Tate module of $E$. 
In this conjugate self-dual case, the phenomena are intertwined with the root number of $E$ over $K$, the splitting of $p$ in $K$ as well as the type of reduction of $E$ at $p$. This article explores anticyclotomic Iwasawa
theory of $E$ at a prime $p$ of non-ordinary reduction in the root number $-1$ case, allowing $p$ to be split or inert in $K$.
It is a sequel to our article \cite{BBL1} which considered the root number $+1$ case. 

Let $N_0$ denote the conductor of $E$.  
Suppose that the discriminant of the imaginary quadratic field $K$ is coprime to $N_0$. Write $N_0 = N^{+}N^{-}$, where $N^+$ and $N^-$ are the positive integers only divisible by primes which are split and 
inert in $K$, respectively. Throughout the paper, we assume the following generalized Heegner hypothesis: 
\begin{itemize}
    \item [(\mylabel{item_GHH}{\textbf{H}})] The integer $N^-$ is a square-free product of an even number of primes. 
\end{itemize}
It implies that the root number of $E$ over $K$ equals $-1$.

Let $p \geq 5$ be a prime of good supersingular reduction for $E$ unramified in $K$. In particular, $a_p(E) = 0$. Let $\rho_E: G_\QQ \ra  \GL_2(\ZZ_p)$ denote the associated $p$-adic Galois representation and $\overline{\rho}_E: G_\QQ \ra \GL_2(\FF_p)$ its mod $p$ reduction. 
Let $K_\infty$ denote the anticyclotomic $\Zp$-extension of $K$ and put $\Gamma=\Gal(K_{\infty}/K)$. 
For an integer $m\ge0$, write $K_m$ for the unique subextension of $K_{\infty}/K$ such that $[K_m:K]=p^m$. 
Put $\Gamma_m=\Gal(K_\infty/K_m)$ and $G_m=\Gal(K_m/K)$.
Let $\LL$ denote the Iwasawa algebra $\Zp[[\Gamma]]=\displaystyle\varprojlim_m\Zp[G_m]$. 
Since the root number of $E$ over $K_m$ also equals $-1$, one expects the existence of non-torsion points along the anticyclotomic tower. Hence the anticyclotomic plus/minus Selmer groups \'a la Kobayashi \cite{kobayashi03} are not expected to be $\LL$-cotorsion, 
 unlike the setting of \cite{BBL1}. 
A beautiful idea of Heegner provides a candidate for the non-torsion points via modular parametrization of $E$ and CM theory: Heegner points. 

Iwasawa theory of Heegner points was initiated by Perrin-Riou in the late 80's \cite{perrinriou87}. With an inspiration from Kolyvagin's work, she formulated the Heegner point main conjecture for an ordinary prime 
$p$ split\footnote{In fact Perrin-Riou allows $p$ to be inert in $K$, but the rest of the discussion in the paragraph requires $p$ to be split.} in $K$, 
which relates the maximal $\Lambda$-torsion submodule of the Pontryagin dual of the $p$-primary Selmer group of $E$ over $K_\infty$ to the index of the $\Lambda$-module generated by the Heegner points. 
In his thesis, Bertolini \cite{Bertolini1995} studied one divisibility toward this main conjecture: upper bound for the Selmer group in terms of the $\Lambda$-module generated by Heegner points,  based on  the Kolyvagin system of Heegner points. 
About a decade later, it was refined by Howard \cite{howard04}, completing proof of this divisibility. Subsequently, building on 
seminal ideas of Bertolini and Darmon~\cite{BertoliniDarmon2005}, Howard \cite{howard06} found a criterion under which the Heegner point main conjecture holds true, thereby introducing the framework of bipartite Euler systems. Howard's criterion was recently verified in \cite{BCK}, yielding a proof of the conjecture under some hypotheses. The method relies on Zhang's breakthrough \cite{zhang14} towards  Kolyvagin's conjecture, which asserts non-triviality of the Heegner point Kolyvagin system. A different approach was proposed by Wan \cite{wan21}, relying on the $p$-adic Waldspurger formula \cite{bertolinidarmonprasanna13} and a Rankin--Selberg main conjecture. 


More recently, for a supersingular prime $p$ split in $K$, 
Castella and Wan \cite{castellawan} formulated plus/minus Heegner point main conjectures.
The formulation involves plus/minus Selmer groups \`a la Kobayashi \cite{kobayashi03} (see also \cite{longovigni}). Using the strategy in \cite{wan21}, Castella and Wan also proved the plus/minus Heegner point main conjecture for semistable elliptic curves. 

In this article, we formulate and study plus/minus Heegner point main conjectures for a supersingular prime $p$ {\it{inert}} in $K$. 
The formulation 
is rooted in the work of
the first named author\footnote{He is grateful to Shinichi Kobayashi and Kazuto Ota for inspiring discussions (see also \cite{BHKO,BKO3,BKNO}).} with Kobayashi and Ota \cite{BKO1,BKO2} on anticyclotomic CM Iwasawa theory at inert primes (cf.~Conjecture~\ref{conj:IMC}). The main result proves the minus Heegner point main conjecture for semistable elliptic curves (cf.~Theorem~\ref{thm:main}). It has an application to the Birch and Swinnerton-Dyer conjecture: $p$-converse to the Gross--Zagier and Kolyvagin theorem for supersingular primes $p$ (cf.~Corollary \ref{p-cv}).
Our approach also treats the case of general weight two elliptic newforms and non-ordinary primes $p$ split in $K$ (see Section~\ref{appA}). 

\subsection{Main results}
We present the framework and results for an elliptic curve and a supersingular prime $p$ inert in $K$, referring the reader to \S\ref{appA} for the split non-ordinary case.

Our framework begins with definition of an anticyclotomic counterpart 
$\Sel_\pm(K_\infty,E)$ 
of Kobayashi's plus/minus Selmer groups in the cyclotomic setting (cf.~\S\ref{sec:Sel}). 
It relies on a system of local points introduced in \cite{BKO1}, which leads to plus/minus local conditions at the prime above $p$. We then define compact plus/minus Selmer groups $H^{1,\pm}(K,\TT^\ac) \subset \displaystyle\varprojlim_m H^1(K_m,T_p E)$, where $T_p E$ denotes the $p$-adic Tate module of $E$ and $\TT^\ac:=T_pE \otimes \Lambda$. 

In view of the generalized Heegner hypothesis \eqref{item_GHH},
the elliptic curve $E$ admits parametrization by a Shimura curve  associated with the indefinite quaternion algebra of discriminant $N^-$. The images of CM points on the Shimura curve with CM by an order of $K$ of conductor $p^n$ under the parametrization  define Heegner points on $E$ along the anticyclotomic $\ZZ_p$-tower. 
In \S\ref{subsec_8_1_2022_09_27_0906} we introduce a certain $\Lambda$-adic linear combination of the Heegner points, 
giving rise to plus/minus Heegner classes 
$$\kappa(1)^\pm\in H^{1,\pm}(K,\TT^\ac).$$

\subsubsection{Main conjectures}
In the spirit of main conjectures of Perrin-Riou \cite{perrinriou87} and Kobayashi \cite{kobayashi03} we propose the following. 

\begin{conj}[plus/minus Heegner point main conjectures]\label{conj:IMC} 
Let $E$ be an elliptic curve defined over $\QQ$ of conductor $N_{0}$ and $p\geq 5$ a prime of good supersingular reduction. Let $K$ be an imaginary quadratic field with discriminant coprime to $N_0$ so that the generalized Heegner hypothesis \eqref{item_GHH} holds and $p$ is inert in $K$.  
\begin{itemize}
    \item[a)] For $\bullet\in\{+,-\}$, we have 
    $\rank_\Lambda\,\Sel_\bullet(K_\infty,E)^\vee = 1$\,.
 \item[b)] Let $\Sel_\bullet(K_\infty,A_f)^\vee_{\rm tor}$ denote the $\LL$-torsion submodule of $\Sel_\bullet(K_\infty,A_f)^\vee$. Then we have\footnote{Here $\Char_\Lambda(M)$ denotes the characteristic ideal of a  $\Lambda$-module $M$.}
\[
{\rm char}_{\LL}\left(H^{1,\bullet}(K,\TT^\ac)/\LL\cdot\kappa(1)^\bullet \right)^2 = {\rm char}_{\LL}(\Sel_\bullet(K_\infty,E)^\vee_{\rm tor})
\]
as ideals of $\LL[1/p]$. Moreover, for $\bullet=-$, the equality holds in $\LL$.     
\end{itemize}
    \end{conj}

\begin{remark} In ~\cite{kobayashi03,castellawan} both the plus and minus main conjectures are expected to hold integrally. Our setting exhibits a new phenomenon (see \S\ref{ss:rmk}). 
\end{remark}

\subsubsection{Results}
Our results are conditional on the following hypotheses:
\begin{itemize}
        \item[(\mylabel{item_Loc}{\textbf{ram}})]
        If $\ell\mid N_0$, then $\bar{\rho}_E$ is ramified at $\ell$ in either of the following cases:
        \begin{itemize}
        \item[\tiny{$\circ$}] $\ell\mid N^{+}$,
        \item[\tiny{$\circ$}] $\ell \mid N^-$ and $\ell^2\equiv 1\mod{p}$.
        \end{itemize}
    \end{itemize}
    
The main result of this article 
is the following.
\begin{theorem}
\label{thm:main}
Let $E$ be a semistable elliptic curve defined over $\QQ$ of  conductor $N_{0}$ and $p\geq 5$ a prime of good supersingular reduction. Let $K$ be an imaginary quadratic field with discriminant coprime to $N_0$ and satisfying the generalized Heegner hypothesis \eqref{item_GHH}. Suppose that $p$ is inert in $K$, and satisfies the hypothesis 
\eqref{item_Loc}.
\begin{itemize}
    \item[i)] For $\bullet\in \{+,-\}$, we have $\rank_\Lambda\, H^{1}_{\bullet}(K,\TT^\ac)=\rank_\Lambda\,\Sel_\bullet(K_\infty,E)^\vee=1$\,. 
    \item[ii)] We have $${\rm char}_{\LL}\left(H^{1}_{+}(K,\TT^\ac)/\LL\cdot\kappa(1)^+ \right)^2 \subset {\rm char}_{\LL}(\Sel_+(K_\infty,E)^\vee_{\rm tor}).$$ 
    \item[iii)] For $\bullet=-$, Conjecture~\ref{conj:IMC} holds true, that is
    $${\rm char}_{\LL}\left(H^{1}_{-}(K,\TT_{f}^\ac)/\LL\cdot\kappa(1)^- \right)^2 = {\rm char}_{\LL}(\Sel_-(K_\infty,E)^\vee_{\rm tor})\,$$
    as ideas of $\Lambda$.
\end{itemize}
\end{theorem}
\begin{remark}
Theorem~\ref{thm:main} iii) may be restated as follows:
We have 
 $${\rm char}_{\LL}\left(H^{1}_{-}(K,\TT_{f}^\ac)/\LL\cdot\kappa(1)^- \right)^2 \subset {\rm char}_{\LL}(\Sel_-(K_\infty,E)^\vee_{\rm tor})\,$$
 and the equality holds if the signed \emph{cyclotomic} main conjecture holds for weight two elliptic newforms\footnote{In fact, our method requires it for a specific weight two newform: a newform obtained from mod $p$ level raising of the newform associated with $E$ for which the $p^\infty$-Selmer group over $K$ vanishes.} with square-free conductor and good non-ordinary reduction at $p$ or their quadratic twist with good reduction at $p$. Such cases of the main conjecture are established in \cite{bstw,CCSS}.
\end{remark}

We have the following application of the $p$-converse to the Gross--Zagier and Kolyvagin theorem \cite{GZ86,koly}.  
\begin{corollary}\label{p-cv}
Let $E$ be a semistable elliptic curve defined over $\QQ$ and $p\geq 5$ a prime of good supersingular reduction. Suppose that $p$ is coprime to the Tamagawa number of $E$. 
Then, 
$$
{\rm corank}_{\ZZ_{p}}\Sel_{p^\infty}(E/\QQ)=1 \quad \implies \quad \ord_{s=1}L(E,s)=1.
$$
\end{corollary}
\begin{proof}
Pick an imaginary quadratic field $K$ with $p$ inert, all primes dividing $N_0$ split so that 
$$
L(E\otimes \epsilon_K,1) \neq 0,
$$
 where $\epsilon_K$ denotes the quadratic character associated with the extension $K/\QQ$ (cf.~\cite{bf}). 
 
 Note that \eqref{item_Loc} is satisfied due to our hypothesis on the Tamagawa number. Therefore, Theorem~\ref{thm:main}(iii) holds. It implies the $p$-converse by the same argument as in the proof of \cite[Theorem 6.11]{castellawan}.
\end{proof}

\begin{remark}
A decade ago, Skinner \cite{skinner20} and Zhang \cite{zhang14} initiated the study of $p$-converse to the Gross--Zagier and Kolyvagin theorem. Their pioneering work considered ordinary primes $p$ (see \cite{bst2021} for an overview). 
The above supersingular case is originally due to Castella and Wan \cite{castellawan}, of which our method provides a different proof. 
\end{remark}
\subsection{Strategy}\label{ss:strategy}
We approach Theorem~\ref{thm:main} by adapting Howard's framework of bipartite Euler systems to the present setting. It also relies on Zhang's approach to Kolyvagin's conjecture \cite{zhang14} and a recent progress in cyclotomic non-ordinary Iwasawa theory \cite{bstw,CCSS}. 

We begin with construction of a family of plus/minus Heegner classes and 
$p$-adic $L$-functions in the guise of theta elements,  
and subsequently prove reciprocity laws connecting the two (see \S\ref{subsec_8_1_2022_09_27_0906}). This framework is ancillary to the strategy, and builds on the work of Bertolini--Darmon \cite{BertoliniDarmon2005} and Darmon--Iovita \cite{darmoniovita}, and Burungale--Kobayashi--Ota \cite{BKO1,BKO2}. Next, we show that the plus/minus Heegner classes extend to  $\Lambda$-adic Kolyvagin systems. Consequently, part (i) and the upper bound for Selmer groups in parts (ii)-(iii) of Theorem~\ref{thm:main} follow. 
As for the equality in the plus/minus Heegner point main conjecture, 
our variant of Howard's equality criterion asserts that it holds if the bipartite Euler system arising from the plus/minus Heegner class and theta element is primitive. 

The primitivity is approached via Zhang's principle  of level raising and rank lowering \cite{zhang14}.  
A key for its implementation: a formula that relates the value of the plus/minus theta elements at the trivial character to the algebraic part of the associated central $L$-value.
We establish such a formula for the minus theta element, whereas the plus theta element curiously vanishes at the trivial character! 
In light of Zhang's principle, 
the minus Heegner point main conjecture then reduces to the implication: 
if the mod $p$ Selmer group associated with a certain\footnote{It arises from level raising and rank lowering of the newform associated with $E$, the existence of which is due to Zhang \cite{zhang14}.} weight two elliptic newform $g$ over $K$ is trivial, then algebraic part of the associated central $L$-value is a $p$-adic unit. This implication is a consequence of the signed cyclotomic Iwasawa main conjecture for $g$ and $g\otimes \epsilon_K$ over $\QQ$ as established in \cite{bstw,CCSS}. 

The strategy necessitates working with elliptic newforms $g$ congruent to $f_E$ for which $a_p(g)$ is not necessarily zero, since the method of level raising only guarantees that $a_p(g)\equiv a_p(f_E)= 0 \mod{p}$. This leads to some technical complications. 
For example, we 
consider non-integral theta elements to establish the aforementioned formula for the plus/minus theta elements at the trivial character (see Lemma~\ref{lem:interpolation}).

We now discuss the case where $p$ splits in $K$. 
Generalizing our previous work~\cite{BBL1}, we extend Theorem~\ref{thm:main} to abelian varieties of $\GL_2$-type over $\QQ$ that are non-ordinary at $p$ for primes $p\geq 5$ split in $K$ under similar hypotheses (see Theorem~\ref{thm:appendix}). In this setting $a_p(f)$ is not necessarily zero, where $f$ denotes the  associated weight two newform, and we construct signed Heegner classes and theta elements.
We approach the signed main conjecture via the above strategy. 
As for the value of the signed theta elements at the trivial character, we modify the calculation in the inert case, and the outcome turns out to be different! Namely, the signed theta elements relate to the algebraic part of the central $L$-value for both signs. Here we utilize an explicit Waldspurger formula of Cai--Shu--Tian \cite{cst} (see the proof of Lemma~\ref{lem:app-interpolation}). Consequently, we establish the signed Heegner point main conjectures for both signs and the analogue of Corollary~\ref{p-cv} for $\GL_2$-type abelian varieties (see~Theorem~\ref{thm:appendix} and Corollary~\ref{cor:p-conv}). 
In contrast, 
Theorem~\ref{thm:main} 
is only shown for elliptic curves, since it relies on the anticyclotomic family of local points constructed in \cite{BKO1,BKO2}, which is not yet generalized to $\GL_2$-type abelian varieties (see Remark~\ref{rk:lastwords}).

The strategy above has its roots in the work of the first-named author with Castella and Kim \cite{BCK}, which established Perrin-Riou's Heegner point main conjecture at ordinary primes in several cases. It notably differs from the work of Castella--Wan \cite{castellawan}. In {\it{loc. cit.}} the authors rely on the $p$-adic Waldspurger formula of Bertolini--Darmon--Prasanna \cite{bertolinidarmonprasanna13} relating the plus/minus Heegner points with an anticyclotomic $p$-adic $L$-function and an Eisenstein congruence divisibility \cite{clw}  towards the main conjecture for this $p$-adic $L$-function. On the other hand, we rely on a cyclotomic main conjecture in light of Zhang's aforementioned principle of level raising and rank lowering. 

\subsubsection{Remarks on Conjecture~\ref{conj:IMC}}\label{ss:rmk}
An integral formulation of plus Heegner point main conjecture is an interesting open problem. 

Due to the exceptional zero phenomenon for the plus theta element explained above (see also Lemma~\ref{lem:interpolation}), 
the plus bipartite Euler system to which the Heegner class $\kappa(1)^+$ belongs does {\it{not}} satisfy the equality criterion of Howard. In turn, our method does not suggest an optimal form for the plus Heegner point main conjecture. We surmise that the naive formulation is not true. This rudimentary guess is based on the following principle\footnote{The principle  was first proposed by Mazur and Rubin \cite{mr02} 
in the context of Kolyvagin systems over discrete valuation rings (a rigidity phenomenon), and later generalized to cyclotomic deformations (resp. deformations over complete regular local rings) by the second-named author \cite{kbb} (resp.~\cite{kbbdeform}; see also \cite{sakamoto_Gorenstein,KimKimSun}). 
For an elliptic newform, it implies that the Kolyvagin system attached to the residual Galois representation is non-zero if the bound arising from Beilinson--Kato elements is optimal. The latter is known to hold under some hypotheses (cf.~\cite{skinnerurbanmainconj}).}: the bound arising from a Kolyvagin system over a complete local Noetherian ring is optimal if and only if the Kolyvagin system is non-zero modulo the maximal ideal.

\subsection{Related works}
This work was first announced in \cite{BBL1}.
During the final stages of its preparation, we became aware of the recent work of Bertolini--Longo--Venerucci \cite{BLV}. 
In {\it{loc. cit.}} the authors establish the anticyclotomic main conjecture in new cases, which include Theorem~\ref{thm:main}. As for the split case, {\it{loc. cit.}} assumes that $a_p(f)=0$ in the non-ordinary case, whereas we treat the general non-ordinary case (cf.~\S\ref{appA}). 
 The strategy in {\it loc. cit.} shares some similarities as well as differences with this article. In the inert supersingular case, both crucially rely on the construction of local points in \cite{BKO1}. 
 Our proof of Theorem~\ref{thm:main}(iii) employs the equality criterion originating from~\cite{howard06}, where Howard formalized the method of Bertolini and Darmon~\cite{BertoliniDarmon2005}.  On the other hand, Bertolini--Longo--Venerucci refine the inductive argument of \cite{BertoliniDarmon2005}. Our treatment of the interpolation formulas for the plus/minus theta elements and their utility also differ from \cite[\S4]{BLV} (see~\S\ref{ss:strategy}).  
As illustrated in \S\ref{appA}, our framework is more readily adaptable to the $a_p(f)\ne0$ case. 

In a different direction, when $p$ splits in $K$, the method of \cite{BCK} has been very recently used by Castella--Hsu--Kundu--Lee--Liu to prove the plus/minus Heegner point main conjecture for elliptic curves \cite[Appendix A]{CHKLL}. Note that Theorem \ref{thm:main} and the results in \S\ref{appA} complement {\it{loc. cit.}}

Our explicit formulas for the value of plus/minus theta elements at the trivial character have other applications to supersingular Iwasawa theory (see~\cite{CHKLL}, Remark~A.7). 

\subsubsection{Horizon} In the inert case, the search for a $p$-adic $L$-function or a $p$-adic Waldspurger formula which is connected with Selmer groups and leads to an Iwasawa main conjecture remains  elusive. 
An important problem is to find connections between our plus/minus Heegner points and the $p$-adic Waldspurger formula of Andreatta and Iovita \cite{AndreattaIovitaBDP}.  

\subsection{Organization} Sections \ref{S:notation}--\ref{sec:Sel} present some preliminaries. In particular, section \ref{sec:coleman} constructs the relevant Coleman maps using local results of \cite{BKO1,BKO2}. Section \ref{S:Shimura} constructs plus/minus Heegner points and theta elements, and establishes their basic properties. Then section \ref{sec_Heeg_classes_construction_reciprocity} adapts Howard's framework of bipartite Euler systems to the inert supersingular case, and verifies the analogue of his equality criterion, leading to Theorem~\ref{thm:main}. Finally, section \ref{appA} treats the split non-ordinary case. 

\subsection*{Acknowledgements}
We thank Francesc Castella, Henri Darmon, Adrian Iovita, Chan-Ho Kim, Shinichi Kobayashi,  Kazuto Ota, Chris Skinner,
Naomi Sweeting, 
Matteo Tamiozzo, Ye Tian and Xin Wan for helpful discussions. 
We are grateful to the referees for their extensive comments,
acute questions 
and constructive suggestions.

AB is partially supported by the NSF grants DMS-2303864 and DMS-2302064. KB’s research in this publication was conducted with the financial support of Taighde \'{E}ireann -- Research Ireland under Grant number IRCLA/2023/849 (HighCritical). AL's research is supported by the NSERC Discovery Grants Program RGPIN-2026-04351. Parts of this article are based upon work supported by the National Science Foundation under Grant No. DMS-1928930 while the authors were in residence at the MSRI / Simons Laufer Mathematical Sciences Institute (SLMath) in Berkeley, California, during the Spring 2023 semester.  

\section{Notation}\label{S:notation}

Let the setting be as in the introduction. 
In particular, $K$ denotes an imaginary quadratic field and $p\geq 5$ a prime, which is assumed to be inert in $K$ until
\S\ref{sec_Heeg_classes_construction_reciprocity}.
Fix an elliptic curve $E/\QQ$ of conductor $N_0$ with good supersingular reduction at $p$ and let $T_p(E)$ denote the associated $p$-adic Tate module. The hypothesis  \eqref{item_Loc} will be assumed from \S\ref{S:Shimura} onward.

Let $f\in S_2(\Gamma_0(N_0))$ be the weight two elliptic newform associated with $E$ and let $T_f$ denote the $f$-isotypic subquotient of $H^1_{\text{\'et}}(X_0(N_0),\ZZ_p)$. 
An optimal modular parameterization of $E$ induces an isomorphism $T_p(E)\simeq T_f$.
Put $V_f=T_f\otimes_{\Zp}\Qp$ and $$A_f=V_f/T_f\simeq E[p^\infty].$$ Given an integer $n\ge0$,  write $\Tfn$ for $T_f/p^n T_f\simeq E[p^n]$ and put\footnote{Even though $\Tfn$ and $\Afn$ are isomorphic as Galois modules, we find it helpful to distinguish them, especially when considering duality.} $A_{f,n}:=A_f[p^n]\simeq E[p^n]$.

Throughout this article, we consider weight two Hecke eigenforms on a quaternion algebra that are congruent to $f$.
Let $$X=X_{M^+,M^-}$$ be the Shimura curve (resp.~Shimura set) {attached to an indefinite (resp.~a definite) quaternion algebra of discriminant $M^-$ and of level $\Gamma_0(M^+)$ (see \cite[\S1.3]{BD-mumford} for a more detailed discussion)}. Let $h$ be a $\Zp$-valued weight two Hecke eigenform on $X$. Write $T_h$ for the $\Zp$-linear $G_\QQ$-representation associated with $h$, and define $T_h$, $V_h$, $A_h$ and $\Thn$ just as above.

For a field $k$, write $G_k=\Gal(\overline{k}/k)$, where $\overline{k}$ denotes a separable closure. 
For an algebraic extension $k'/k$ and  a $G_k$-representation $W$, let
$$ \cor_{k'/k}:H^1(k',W)\rightarrow H^1(k,W)\quad\text{and}\quad \res_{k'/k}:H^1(k,W)\rightarrow H^1(k',W)$$
be the corestriction and restriction maps respectively. For a $p$-adic Lie extension $\mathcal{K}/k$, put
\[
H^i_{\mathrm{Iw}}(\mathcal{K},W)= \varprojlim H^i(k',W),
\]
where the inverse limit is over finite extensions $k'$ of $k$  contained in $\mathcal{K}$ with respect to corestriction maps.


\section{Coleman maps}\label{sec:coleman}
The aim of this section is to introduce the pertinent local Iwasawa theory.

Recall that $p$ is inert in the imaginary quadratic field $K$ and 
$K_\infty$ denotes the anticyclotomic $\Zp$-extension of $K$.
Throughout this section, let $w$ be a place of $K_\infty$ lying above $p$. Let $$\Gamma_w=\Gal(K_{\infty,w}/K_p)\cong\Zp$$ be the decomposition group of $w$ in $\Gamma$. 
Note that $K_p=\QQ_{p^2}$. 
To lighten the notation, we shall write $k_\infty=K_{\infty,w}$. For an integer $ m\ge0$, let $k_m$  denote the intermediate extension of $k_\infty/K_p$ that is of degree $p^m$. Let $\cG_m=\Gal(k_m/k_0)$.

Put
$$\mathfrak{R}=\cO_{K_p}[[\Gamma_w]]=\varprojlim\cO_{K_p}[\cG_m].$$ 
Given an integer $n\ge 1$,  write $\mathfrak{R}_n=\mathfrak{R}/p^n \mathfrak{R}=(\cO_{K_p}/p^n\cO_{K_p})[[\Gamma_w]]$. If $m\ge0$ is another integer, let $\mathfrak{R}_{m,n}$ denote the group ring $(\cO_{K_p}/p^n\cO_{K_p})[\cG_m]$.

\subsection{Local points}
The following result of Burungale--Kobayashi--Ota from \cite{BKO1} is key to this section.

\begin{theorem}\label{thm:Q-inert}
There exists a system of local points $d_m\in \hat{E}(\fM_{k_m})$ such that:
\begin{itemize}
    \item[(1)] $\Tr_{k_m/k_{m-1}}d_m=-d_{m-2}$  for all $m\ge 2$;
    \item[(2)] $\Tr_{k_1/k_0}d_1=-d_{0}$;
    \item[(3)] $d_0\in \hat{E}(\fM_{k_0})\setminus p\hat{E}(\fM_{k_0})$.
\end{itemize}
\end{theorem}
\begin{proof}
This is \cite[Theorem~5.5]{BKO1}. 
\end{proof}

Put
\[
d_m^+=\begin{cases}
d_m&\textrm{if $m$ is even},\\
d_{m-1}&\textrm{if $m$ is odd},
\end{cases}\qquad d_m^-=\begin{cases}
d_{m-1}&\textrm{if $m\ge2$ is even},\\
d_{m}&\textrm{if $m$ is odd}.
\end{cases}\quad
\]

The formal group $\hat E$ is equipped with a natural $\cO_{K_p}$-action, which leads to the following definition of plus and minus subgroups inside $\hat{E}(k_m)$.

\begin{defn}
Let $m\ge0$ be an integer. We define $\hat E^\pm(k_m)$ to be the $\mathfrak{R}$-modules generated by $d_m^\pm$, respectively. 
\end{defn}
\begin{remark}
These subgroups can be described in terms of the trace maps, as in \cite[Definition~8.16]{kobayashi03}.    
\end{remark}

We may identify $\hat E(k_m)$ (resp. $\hat E(k_m)\otimes \Qp/\Zp$) with $H^1_{\rm f}(k_m,T_f)$ (resp. $H^1_{\rm f}(k_m,A_f)$) via the Kummer map.  Analogous to Lemma~8.17 of \textit{op. cit.}, the Kummer map also induces an injection 
\[
\hat E^\pm(k_m)\otimes\Qp/\Zp\hookrightarrow H^1(k_m,A_f).
\]
Furthermore, as in \cite[Proposition~8.6]{kobayashi03}, $\hat E(k_m)$ is $p$-torsion-free. Thus,
\[
H^1(k_m,\Afn)=H^1(k_m,A_f)[p^n]
\]
and we may identify $\hat E^\pm(k_m)/(p^n)\simeq \left(\hat E^\pm(k_m)\otimes\Qp/\Zp\right)[p^n]$ as a subgroup of $H^1(k_m,\Afn)$. 

\subsection{Coleman maps} 
In view of the last paragraph of the preceding subsection, proceeding as in \cite[Proposition~4.28]{BBL1}, we may define the $\mathfrak R$-morphisms
 $$\col^{\pm}_{\Tfn}:\HIw(k_\infty,\Tfn)\lra \mathfrak{R}_n.$$ 
More explicitly, $\col^\pm_{\Tfn}$ can be described as follows. 

We identify $\mathfrak{R}_n$ with the ring of power series $\cO_{K_p}/(p^n)[[X]]$. Let $m\ge1$ be an integer. Write $\omega_m=(1+X)^{p^m}-1\in \mathfrak{R}_n$ and let $\Phi_m$ denote the $p^m$-th cyclotomic polynomial. 
 Let
\[\langle-,-\rangle_{m,n}:H^1(k_m,\Tfn)\times H^1(k_m, \Afn)\lra \cO_{K_p}/(p^n).\]
denote the local $\cO_{K_p}$-linear pairing.
Define
 \begin{align*}
 \col^\pm_{m,n}:H^1(k_m,\Tfn)&\lra \mathfrak{R}_{m,n} , \\
 z&\longmapsto \sum_{\sigma\in G_m}\langle z^{\sigma^{-1}},d_m^\pm\rangle_{m,n}\sigma.
 \end{align*}
It can be shown as in \cite[Proposition~8.19]{kobayashi03} that the image of $\col^\pm_{m,n}$ lies inside 
\begin{equation}
\label{eqn_2023_04_17_1459}
    \mathfrak{R}_{n}/(\omega_m^\mp)\xrightarrow[\times \widetilde{\omega}_m^\pm]{\,\sim\,} \widetilde \omega_m^\pm \mathfrak{R}_{m,n}\,,
\end{equation}
where $$\widetilde\omega_m^+=\prod_{2\le r\le m,r\ \mathrm{even}}\Phi_r(1+X),\quad \widetilde\omega_m^-(X)=\prod_{1\le r\le m,r\ \mathrm{odd}}\Phi_r(1+X),\quad\text{and}\quad \omega_m^\pm=\omega_m/\widetilde\omega_m^\mp.$$
Taking inverse limit then results in the maps $\col^\pm_{\Tfn}$.

From this description, one can see that the kernel of $\col^{\pm}_{\Tfn}$ is precisely $\displaystyle\varprojlim_m (\hat E^\pm(k_m)/(p^n))^\perp$, where $(\hat E^\pm(k_m)/(p^n))^\perp\subset H^1(k_m,\Tfn)$ is the orthogonal complement of $\hat E^\pm(k_m)/(p^n)\subset H^1(k_m,\Afn)$ under $\langle-,-\rangle_{m,n}$.
 Note that even though $\col^\pm_{\Tfn}$ depends on the choice of $d_m^\pm$, its kernel does not.


\begin{defn}
For $\bullet\in\{+,-\}$ and $m\ge0$, define $H^{1,\bullet}(k_m,\Tfn)\subset H^1(k_m,\Tfn)$ to be the image of $\ker\col^\bullet_{\Tfn}$ under the natural projection $\HIw(k_\infty,\Tfn)\rightarrow H^1(k_m,\Tfn)$.

We define $H^1_\bullet(k_m,\Afn)\subset H^1(k_m,\Afn)$ to be the orthogonal complement of $H^{1,\bullet}(k_m,\Tfn)$ under  $\langle -,-\rangle_{m,n}$.
\end{defn}

\begin{remark}
    We have the inclusion $\hat E^\pm(k_m)/(p^n)\subset H^1_\pm (k_m,\Afn)$; see \cite[Corollary~4.29]{BBL1}. When $0<m<\infty$, this inclusion is strict, that is, the two subgroups are not equal to each other, as discussed in Remark~4.9 of \textit{op. cit.}
\end{remark}


\section{Selmer groups}\label{sec:Sel}
The aim of this section is to introduce some of the underlying Selmer groups. 

Recall that $K_\infty$ is the anticyclotomic $\Zp$-extension of $K$. For an integer $m\ge0$, let $K_m\subset K_\infty$ denote the unique subextension such that $[K_m:K]=p^m$. Let $n$ be a positive integer. For a rational prime $\ell$ and $X\in\{A,T\}$, let 
\[
K_{m,\ell}:=K_m\otimes_\QQ\QQ_\ell,\quad H^1(K_{m,\ell},X_{f,n}):=\bigoplus_{\lambda|\ell}H^1(K_{m,\lambda},X_{f,n}),
\]
where the direct sum runs over all primes of $K_m$ above $\ell$. We have the natural restriction map
\[
\res_\ell:H^1(K_m,X_{f,n})\lra H^1(K_{m,\ell},X_{f,n}).
\]

Denote by $H^1_\fin(\Kml,T_{f,n})$  the image of the Bloch--Kato subgroup $H^1_\fin(\Kml,T_f)$ given in \cite[(3.7.3)]{blochkato}. We similarly define $H^1_\fin(\Kml,A_{f,n})$. The singular quotient is given by
\[
H^1_\sing(\Kml,X_{f,n}):=\frac{ H^1(\Kml,X_{f,n})}{H^1_\fin(\Kml,X_{f,n})}.
\]

\begin{defn}
The Bloch--Kato Selmer group of $A_{f,n}$ (resp. $\Tfn$) over $K_m$ is defined as 
\[
\Sel(K_m,A_{f,n}):=\ker\left(H^1(K_m,A_{f,n})\lra \prod_\ell H^1_\sing(\Kml,A_{f,n})\right)\,,
\]
\[
H^1_{\rm f}(K_m,T_{f,n}):=\ker\left(H^1(K_m,T_{f,n})\lra \prod_\ell H^1_\sing(\Kml,T_{f,n})\right).
\]
Put
\[
\Sel(K_\infty,A_{f,n}):=\varinjlim_m \Sel(K_m,A_{f,n})\,,
\]
\[
\widehat H^1(K_\infty,T_{f,n}):=\varprojlim_m H^1_{\rm f}(K_m,T_{f,n}).
\]
Moreover, 
for $\cdot\in\{0, \emptyset\}$  
define $\Sel_\cdot$ and $H^1_\cdot$  by replacing $H^1_\sing(K_{m,p},X_{f,n})$ with $H^1(K_{m,p},X_{f,n})$ and $0$, respectively.
\end{defn}

\begin{defn}
\label{def:signedSelmer} 
For $\bullet\in\{+,-\}$, 
define the $\bullet$-Selmer group by 
\[
\Sel_\bullet(K_m,A_{f,n}):=\ker\left(H^1(K_m,A_{f,n})\rightarrow \prod_{\ell\nmid p} H^1_\sing(\Kml,A_{f,n})\times \prod_{\fp| p}\frac{ H^1(K_{m,\fp},A_{f,n})}{H^1_\bullet(K_{m,\fp},A_{f,n})}\right)\,,
\]
\[
H^{1,\bullet}(K_m,T_{f,n}):=\ker\left(H^1(K_m,T_{f,n})\rightarrow \prod_{\ell\nmid p} H^1_\sing(\Kml,T_{f,n})\times \prod_{\fp| p}\frac{ H^1(K_{m,\fp},T_{f,n})}{H^{1,\bullet}(K_{m,\fp},T_{f,n})}\right).
\]
Moreover, define
\[
\Sel_\bullet(K_\infty,A_{f,n}):=\varinjlim_m \Sel(K_m,A_{f,n})\,,\qquad \Sel_\bullet(K_\infty,A_{f}):=\varinjlim_{n} \Sel_\bullet(K_\infty,A_{f,n}).
\]
\[
\widehat H^{1,\bullet}(K_\infty,T_{f,n}):=\varprojlim_m H^{1,\bullet}(K_m,T_{f,n})\,.
\]

For $0\le m\le\infty$ and $\cdot\in\{0,\emptyset, +,-\}$, we similarly define the Selmer groups  $\Sel_\cdot(K_m,A_f)$.
\end{defn}

Finally, we introduce the following auxiliary Selmer groups.
\begin{defn}
\label{defn_4_3_2024_12_10}
    Let $S$ be a square-free integer that is coprime to $pN_0$. For $\cdot\in\{0,\emptyset,+,-,\}$ define the generalized Selmer group $\Sel_{S,\cdot}(K_m,\Afn)$ 
by
$$\Sel_{S,\cdot}(K_m,\Afn):=\ker\left(\Sel_\cdot(K_m,\Afn)\lra \bigoplus_{\ell \mid S} H^1(K_{m,\ell},\Afn) \right)\,.$$
Similarly, define $H^1_{S,\cdot}(K_m,\Tfn)$ by
$$H^1_{S,\cdot}(K_m,\Tfn):=\ker\left(H^1(K_m,\Tfn)\lra  \bigoplus_{\fp\mid p} \frac{H^1(K_{m,\fp},\Tfn)}{H^1_\cdot(K_{m,\fp},\Tfn)}\,\oplus\,\bigoplus_{\ell\nmid S} \frac{H^1(\Kml,\Tfn)}{H^1_\fin(\Kml,\Tfn)}  \right)\,,$$
where $H^1_\cdot(K_{m,\fp},\Tfn)$ equals $0$ (resp. $H^1(K_{m,\fp},\Tfn)$)
for $\cdot=0$ (resp. $\cdot=\emptyset$). Likewise, put 
$$H^{1,\bullet}_S(K_m,\Tfn):=\ker\left(H^1(K_m,\Tfn)\lra  \bigoplus_{\fp\mid p} \frac{H^1(K_{m,\fp},\Tfn)}{H^{1,\bullet}(K_{m,\fp},\Tfn)}\,\oplus\,\bigoplus_{\ell\nmid S} \frac{H^1(\Kml,\Tfn)}{H^1_\fin(\Kml,\Tfn)}  \right)\,,\quad  \bullet\in \{+,-\}\,.$$
\end{defn}


\section{Shimura curves, Shimura sets, and special loci}\label{S:Shimura}
This section  adapts basic constructions in the work of Bertolini and Darmon ~\cite{BertoliniDarmon2005}
to our inert supersingular setting. 
It also incorporates Zhang's principle of level raising and rank lowering \cite{zhang14}.

As before, let $f=f_E$ be the elliptic newform associated with an  
elliptic curve $E/\QQ$ and $p\geq 5$ a supersingular prime. {From this section onward, we assume that $E$ is semistable, and the hypothesis \eqref{item_Loc}. Since $p\geq 5$, the $G_{\QQ}$-action on $E[p]$ is surjective  (cf.~\cite[Proposition~2.1]{edi97}), which will be used in our arguments.}

\subsection{Setting}

\begin{defn}
\label{def_admissible_prime_mod_pn}
For an integer $n\geq 1$, 
a rational prime $\ell$ is said to be \textbf{$n$-admissible} relative to $f$ if it satisfies the following conditions:
\begin{itemize}
    \item[i)] $\ell\nmid pN_0$ ;
    \item[ii)] $\ell$ is inert in $K$;
    \item[iii)] $p\nmid \ell^2-1$;
    \item[iv)] $p^n$ divides $(\ell+1)^2-a_\ell(f)^2$.
    \end{itemize}
\end{defn}
For an $n$-admissible prime $\ell$, denote by $\epsilon_\ell(f)\in\{\pm 1\}$ the integer so that $p^n$ divides $\ell+1-\epsilon_\ell(f)a_\ell(f)$. Let $\epsilon_\ell:\QQ^\times\to \{\pm 1\}$ be the unique character given by $\ell\mapsto \epsilon_\ell(f)$.

Write $\fL_n$ for the set of $n$-admissible primes and $\fN_n$ that of of square-free products of primes in $\fL_n$. We will consider $1$ as an element of $\fN_n$.
For $1\neq S\in \fN_n$,  
let $S$ also denote the set of prime factors of $S$.
  As before, let $\epsilon_K$ denote the quadratic character associated with the extension $K/\QQ$.
  
\begin{defn}
   We say that $S\in\fN_n$ is \textbf{definite} if $${\epsilon_K(SN^-)=-1.}$$ 
    Otherwise, it is said to be \textbf{indefinite}. 
\end{defn}
  
    Let $\fN_n^\de$ (resp. $\fN_n^\inde$) denote the subset of elements in $\fN_n$ that are definite (resp. indefinite). 
    \begin{example} $1\in \fN_n^\inde$. 
    \end{example}
Following \cite[Definition 2.22]{BertoliniDarmon2005} and \cite[Definition 3.16]{darmoniovita}, we say that $S\in \fN_n$ is \emph{\textbf{core-$n$-admissible}} if ${\rm Sel}_{S,\emptyset}(K,A_{f,n})=0$.

\begin{proposition}
\label{prop_2023_04_17_1447}
Let $n$ be a positive integer and $S \in \fN_n$. Then there exists a core-$n$-admissible $S' \in \fN_n$ divisible by $S$.
\end{proposition}
\begin{proof} In view of surjectivity of the $G_{\QQ}$-action on $E[p]$ (cf.~\cite[Proposition~2.1]{edi97}), the existence follows from 
\cite[Theorem 3.2]{BertoliniDarmon2005}. As remarked in the proof of \cite[Proposition~7.9]{BBL1}, neither the choice of local conditions at $p$ nor the splitting behaviour of $p$ in $K/\QQ$  plays a role in the argument. 
\end{proof}

\begin{remark}\label{rmk:bd} For the existence of admissible primes under mild hypotheses, the reader may refer to \cite[Theorem 3.3.5 and Corollary 3.3.9]{sweeting}.
It may eventually lead to removal of the hypothesis that $E$ is semistable 
in our results (see also Remark~\ref{rmk:lr}).
\end{remark}

\subsection{Shimura curves}
\label{subsection_shimura_curves}
Fix an integer $$S\in 
\bigcup_n \fN_n^\inde$$ and let $B_S$ denote the indefinite quaternion algebra over $\QQ$ of discriminant $N^-S$. 

We consider the compact Shimura curve 
$$X_S:=X_{N^+,N^-S}$$
attached to an Eichler order $R_S\subset B_S$ of level $N^+$. Then, 
$$X_S(\CC)=B_S^\times \big\backslash\left( (\CC-\mathbb{R}) \times \widehat{B}_S^\times\right)\big/\widehat{R}_S^\times \,\, \bigcup \, \{\hbox{cusps}\}\,,$$
where
$\widehat{B}_S^\times:=B_S\otimes_{\ZZ}\widehat{\ZZ}$ and $\widehat{R}_S^\times:=R_S\otimes_{\ZZ}\widehat{\ZZ}$. Note that unless $N^-S=1$, the set $\{\hbox{cusps}\}$ is empty.
In the latter case, $X_S=X_0(N)$ is the classical modular curve. 

The curve $X_S$ is a coarse moduli space, and admits a smooth regular model over $\ZZ[\frac{1}{NS}]$, which we still denote by $X_S$ (cf. \cite{KatzMazur85, Buzzard97, Helm07}, see also \cite[\S4.2]{jsw} for a description of the moduli problem). 
\subsubsection{CM locus}
Fix an optimal embedding $\iota_K:K\hookrightarrow B_S$ such that $$\iota_K(K)\cap R_S=\iota_K(\cO_K),$$ where $\cO_K\subset K$ is the ring of integers. Let $\mathfrak{h}$ denote the unique fixed point of the upper half plane under the action of $\iota_K(K)$.  

The set of CM points (relative to the imaginary quadratic field $K$) on $X_S$ is given by 
$${\rm CM}(X_S):=\{[\mathfrak{h},b] \in X_S(\CC):\,b\in \widehat{B}_S^\times\}\xleftrightarrow{\,\sim\,}
K^\times \backslash \widehat{B}_S^\times/\widehat{R}_S^\times\,.$$

\subsubsection{The Heegner point of conductor $p^m$}
\label{subsubsec_522_2023_04_07}

For each natural number $m$, consider the Heegner point $x_m^0 \in {\rm CM}(X_S)$ defined as the image of
$$
\left[\begin{pmatrix}
    p^m & 0\\
0& 1
\end{pmatrix}\right]\in K^\times \big\backslash \widehat{K}^\times{\rm GL}_2 (\QQ_p)\widehat{B}_S^{(p)}\big/\widehat{R}_S^\times \subset K^\times \backslash \widehat{B}_S^\times/\widehat{R}_S^\times \xleftrightarrow{\,\sim\,} {\rm CM}(X_S)\,.$$
The Heegner point $x_m^0 \in {\rm CM}(X_S)$ is defined over the ring class field $K[p^m]$ of conductor $p^m$.

Fix an auxiliary rational prime $q_0 \nmid pNS$ such that $$a_{q_0}(f)-q_0-1 \in \ZZ_p^\times.$$ The existence of $q_0$ is guaranteed by the irreducibility\footnote{Recall that $E$ has good supersingular reduction at $p$.} of the $G_\QQ$-representation $E[p]$.

Let us put
$$x_m:=(a_{q_0}(f)-q_0-1)^{-1}{\rm Tr}_{K[p^{d(m)}]/K_m}\,(T_{q_0}-q_0-1)[x_{m+1}^0]\in {\rm Jac}(X_S)(K_m)\otimes_{\ZZ} \ZZ_p\,,$$
where $T_{q_0}$ is the $q_0$-th Hecke correspondence on $X_S$, $d(m)=\min\left\{d\in\ZZ_{\ge0}: K_m\subset K[p^{d}]\right\}$, and $K_m$ is the $m$-th layer of the anticyclotomic $\ZZ_p$-extension $K_\infty/K$. We define the class 
\begin{equation}
    \label{eqn_2023_3_31_1721}
    z_m \in H^1_{\rm f}(K_m,T_p({\rm Jac}(X_S)))\,.
\end{equation} as the image of $x_m$ under the Kummer map
$${\rm Jac}(X_S)(K_m)\otimes_{\ZZ}\ZZ_p\lra H^1_{\rm f}(K_m,T_p({\rm Jac}(X_S)))\,,$$
where $H^1_{\rm f}(K_m,T_p({\rm Jac}(X_S)))$ denotes the Bloch--Kato Selmer group.

\subsubsection{Arithmetic level raising}
\label{subsubsec_2024_12_04_1546}
Fix a positive integer $n$ so that $S\in \fN_n^\inde$ and denote by $\mathbb{T}_S$ the Hecke algebra generated by the Hecke operators $T_\ell$ for $\ell\nmid N_0S$ and $U_\ell$ for $\ell \mid N_0S$ acting on ${\rm Jac}(X_S)$ over $\ZZ_p$ via Picard functoriality. 

\begin{defn}
    A weak eigenform $h$ on $X_S$ modulo $p^n$ of weight $2$ is a ring homomorphism $$\mathbb{T}_S\xrightarrow{e_h} \ZZ/p^n\ZZ.$$ 
\end{defn}
    
    Let $I_h \subset \mathbb{T}_S$ denote the kernel of $e_h$. The $G_\QQ$-representation attached to $h$ is defined to be 
\begin{equation}
\label{eqn_2023_04_06_1654}
    \Thn:={\rm Jac}(X_S)[p^n]\big{/}I_h{\rm Jac}(X_S)[p^n]\,.
\end{equation}


We have the following key existence (cf.~\cite[\S4.3]{pollack-weston11}, \cite[\S4.2]{chidahsiehanticyclomainconjformodformscomposito}, \cite[Remark 4.5.8]{sweeting}, or \cite[proof of Theorem~5.15]{BertoliniDarmon2005}).
\begin{proposition}\label{prop,indef-existence} Let $E$ be a semistable elliptic curve and $p$ a prime such that the hypothesis \eqref{item_Loc} holds. 
Then for the newform $f=f_E$ and $S\in\fN_n^\inde$,  
there exists a weak eigenform 
\begin{align}\label{eq:f_S}
    \begin{aligned}
    f_S: \TT_S &\lra \ZZ/p^n\ZZ\,,\\
  (\ell\nmid  N_0S) \qquad T_\ell &\longmapsto a_\ell(f) \mod p^n ,\\
(\ell\mid  N_0) \qquad \, U_\ell & \longmapsto a_\ell(f) \mod p^n ,\\
   (\ell\mid S) \qquad  U_\ell &\longmapsto \epsilon(\ell)  \mod p^n ,
    \end{aligned}
\end{align}
where $\epsilon(\ell)=\pm1$ is uniquely determined by the requirement that $p^n$ divides $\ell+1 -\epsilon(\ell)a_\ell(f)$ as in ~Definition~\ref{def_admissible_prime_mod_pn}.
\end{proposition} In turn, we have an isomorphism 
\begin{equation}
\label{eqn_2023_04_06_1212}
    T_{f_S, n}\simeq T_{f,n}
\end{equation}
of $G_\QQ$-representations 
(cf. \cite[Theorem~5.17]{BertoliniDarmon2005}; see also \cite[Proposition~4.4]{pollack-weston11}). 
Fix such an isomorphism for each $S\in \fN_n^\inde$. 

\subsubsection{Selmer groups, bis}
Let $f_S$ be a weak eigenform in the previous subsection. Let $m$ be a natural number and  $\lambda$ a prime of $K_m$ lying above the rational prime $\ell$. 

Define $H^1_\fin(K_{m,\lambda},T_{f_S,n})$ to be the image of $H^1_\fin(K_{m,\lambda},T_{f,n})$ under the isomorphism \eqref{eqn_2023_04_06_1212} and put $$H^1_\fin(\Kml,T_{f_S,n}):=\oplus_{\lambda\mid \ell}\, H^1_\fin(K_{m,\lambda},T_{f_S,n}).$$  

This leads to the Selmer groups $H^1_{S,\emptyset}(K_m,T_{f_S,n})$
associated with the Galois representation $T_{f_S,n}$ as in Definition~\ref{defn_4_3_2024_12_10}.
\begin{remark}
    \label{remark_2024_10_04_1704} The Bloch--Kato subgroup 
    $H^1_{\rm f}(K_{m,\lambda},{\rm Jac}(X_S)[p^n])$
    coincides with the image of the Kummer map:
    $$H^1_{\rm f}(K_{m,\lambda},{\rm Jac}(X_S)[p^n])={\rm im}\left({\rm Jac}(X_S)(K_{m,\lambda})/p^n{\rm Jac}(X_S)(K_{m,\lambda})\xrightarrow{\rm Kum}H^1(K_{m,\lambda},{\rm Jac}(X_S)[p^n])\right)\,.$$
We discuss its relation with $H^1_\fin(K_{m,\lambda},T_{f,n})$ below.
\begin{itemize}
\item[i)] If $\lambda \nmid pN_0S$, then $H^1_\fin(K_{m,\lambda},T_{f_S,n})$ and the image of $H^1_{\rm f}(K_{m,\lambda},{\rm Jac}(X_S)[p^n])$ under the map induced from \eqref{eqn_2023_04_06_1654} coincide with the submodule $H^1_{\rm ur}(K_{m,\lambda},T_{f_S,n})=\ker(H^1(K_{m,\lambda},T_{f_S,n})\to H^1(K_{m,\lambda}^{\rm ur},T_{f_S,n}))$ of unramified cocycles. 
\item[ii)] If $\lambda \mid N_0$, then in view of the hypothesis \eqref{item_Loc}, the conclusion of (i) still holds due to \cite[Lemmas~6 and 8]{grossparson}, which cover the cases when $\lambda$ divides $N^+$ and $N^-$, respectively\footnote{We are grateful to Matteo Tamiozzo for related discussions and also sharing his master thesis, which expands on the exposition of \cite{grossparson}.}. 
\item[iii)] If $\lambda \mid S$, then the image of $H^1_{\rm f}(K_{m,\lambda},{\rm Jac}(X_S)[p^n])$ under the map induced from \eqref{eqn_2023_04_06_1654} does \emph{not} coincide with $H^1_\fin(K_{m,\lambda},T_{f_S,n})$. We refer the reader to \S\ref{subsubsec_517_2024_12_10} for elaboration, particularly the proof of Lemma~\ref{lemma_2023_04_23_2146}.
\item[iv)] If $f_S$ lifts to a characteristic zero Hecke eigenform with associated Galois representation $T_{f_S}$, then one can define $H^1_\fin(K_{m,p},T_{f_S,n})$ directly as the image of the Bloch--Kato subgroup $H^1_\fin(K_{m,p},T_{f_S})$. As explained in \cite[Proposition 6.5]{BBL1}, this alternative definition coincides with the one given above.
\end{itemize}
\end{remark}

\subsubsection{Norm relations}
\label{subsec_3_29_2023}
For a natural number $m$, a positive integer $n$ and $S\in \fN_n^\inde$,  the {\textit{Heegner class}} 
\begin{equation}
\label{eqn_2023_04_23_2203}
    \kappa(S,n)_m\in H^1_{S,\emptyset}(K_m,T_{f,n})
\end{equation}
is defined as the image of the class $z_m$ under the composition of the following morphisms: 
\begin{align}
\begin{aligned}\label{eqn_construct_Heeg_via_level_raising}    
        H^1_{\rm f}(K_m,T_p({\rm Jac}(X_S)))\rightarrow H^1_{\rm f}(K_m,{\rm Jac}(X_S)[p^n]) \xrightarrow{\eqref{eqn_2023_04_06_1654}} H^1_{S,\emptyset}(K_m,T_{f_S,n})\stackrel{\eqref{eqn_2023_04_06_1212}}{\simeq}  H^1_{S,\emptyset}(K_m,T_{f,n})\,. 
    \end{aligned}
\end{align}
The arrow \eqref{eqn_2023_04_06_1654} arises from Remark~\ref{remark_2024_10_04_1704}(i) and (ii).

Note that the class $\kappa(S,n)_m$ does not depend on the choice of the auxiliary prime $q_0$ in \S\ref{subsubsec_522_2023_04_07}. Moreover, the collection $\{\kappa(S,n)_m\}_{m\ge0}$ satisfies the 
trace relation
\begin{equation}
    \label{eqn_Heegner_trace_relation}
    {\cor}_{K_{m+1}/K_{m}}\, \kappa(S,n)_{m+1}= -\res_{K_{m}/K_{m-1}}\kappa(S,n)_{m-1}\,
\end{equation}
for any integer $m\geq 1$, since $a_p(f)=0$.

\subsubsection{} In the special case where $S=1$, we take $f_S=f \mod p^n$. The classes $\kappa (1,n)_m$ are evidently compatible as $n$ varies under the obvious maps. We may therefore define 
$$\kappa(1)_m:=\{\kappa (1,n)_m\}_n\in \varprojlim_n H^1_{\rm f}(K_m, T_{f,n})= H^1_{\rm f}(K_m, T_{f})\,.$$

\subsubsection{Local properties at primes dividing $S$} 
\label{subsubsec_517_2024_12_10}
We close this subsection with an analysis of the local properties satisfied by the class $\kappa(S,n)_m\in H^1_{S,\emptyset}(K_m,T_{f,n})$ at a prime $\ell \mid S$. 

It follows from condition (i) in Definition~\ref{def_admissible_prime_mod_pn} that $T_{f,n}$ is unramified at $\ell$. As explained in \cite[\S2.2]{BertoliniDarmon2005}, the requirements (iii) and (iv) enforce that the action of the Frobenius element over $\QQ$ on $T_{f,n}$ is semisimple with distinct eigenvalues $\pm \ell$ and $\pm 1$. 
This leads to a uniquely determined decomposition
\begin{equation}
\label{eqn_2023_04_23_1856}
T_{f,n}{}_{\vert_{G_{K_\ell}}}\simeq \mu_{p^n} \oplus \ZZ/p^n\ZZ\,.
\end{equation}

\begin{defn}
    \label{define_ord_condition}
    The ordinary local condition 
    $$H^1_{\rm ord}(K_\ell,T_{f,n}) \subset H^1(K_\ell,T_{f,n})$$ 
    on $T_{f,n}$ at $\ell$ is given as the image of $H^1(K_\ell,\mu_{p^n})\to H^1(K_\ell,T_{f,n})$  under the map induced from the decomposition given by \eqref{eqn_2023_04_23_1856}. 
\end{defn}

Since the prime $\ell$ splits completely in $K_m/K$ for any positive integer $m$, we have 
\begin{equation}
    \label{eqn_2023_04_23_2146}
    \bigoplus_{\lambda\mid \ell} H^1(K_{m,\lambda},T_{f,n})=H^1(K_{m,\ell},T_{f,n})=H^1(K_\ell,T_{f,n})\otimes \LL_{m,n}
\end{equation}
for the semi-local cohomology group at $\ell$, where the second identification is due to Shapiro's lemma. Put
$$H^1_{\rm ord}(K_{m,\ell},T_{f,n}):=H^1_{\rm ord}(K_\ell,T_{f,n})\otimes \LL_{m,n}\simeq \bigoplus_{\lambda\mid \ell} H^1_{\rm ord}(K_{m,\lambda},T_{f,n})\,.$$

\begin{lemma}
    \label{lemma_howard_06_221}
    For every positive integer $m$, we have a natural decomposition of free $\LL_{m,n}$-modules
    $$H^1(K_{m,\ell},T_{f,n})=H^1_{\rm ur}(K_{m,\ell},T_{f,n})\oplus H^1_{\rm ord}(K_{m,\ell},T_{f,n})$$
    induced from \eqref{eqn_2023_04_23_1856}, where each summand is maximal isotropic under the local Tate pairing.
\end{lemma}

\begin{proof}
    For $m=1$, this is \cite[Lemma 2.2.1]{howard06}, and the general case follows from \eqref{eqn_2023_04_23_2146}.
\end{proof}

To record the local properties of the Heegner classes, we introduce a triple of Selmer groups.

\begin{defn}
    \label{defn_2023_04_23_2146}
    We define the Selmer group
    $$H^1_{{\rm ord}_S,\emptyset}(K_m, \Tfn):=\ker\left(H^1_{S,\emptyset}(K_m,\Tfn)\lra \bigoplus_{\ell\mid S} \frac{H^1(\Kml,\Tfn)}{H^1_{\rm ord}(\Kml,\Tfn)}  \right)$$
    as the collection of cohomology classes in $H^1_{S,\emptyset}(K_m,\Tfn)$ that satisfy the ordinary condition at primes contained in $S$. We similarly define  $H^{1,\bullet}_{{\rm ord}_S}(K_m, \Tfn)$ for $\bullet\in\{+,-\}$.

Furthermore, we set
$$\widehat{H}^1_{{\rm ord}_S,\emptyset}(K_\infty,\Tfn):=\varprojlim_m {H}^1_{{\rm ord}_S,\emptyset}(K_m,\Tfn)\,,$$
and similarly define $\widehat{H}^{1,\bullet}_{{\rm ord}_S}(K_\infty,\Tfn)$ for $\bullet\in\{+,-\}$.
\end{defn}

\begin{lemma}
        \label{lemma_2023_04_23_2146}
We have
$$\kappa(S,n)_m\in H^1_{{\rm ord}_S,\emptyset}(K_m, \Tfn)\,.$$
\end{lemma}

\begin{proof}
In view of \eqref{eqn_2023_04_23_2203} it suffices to check that 
$$\res_\lambda(\kappa(S,n)_m)\in H^1_{\rm ord}(K_{m,\lambda},\Tfn)=H^1_{\rm ord}(K_\ell,\Tfn), \qquad \forall \lambda\mid \ell.$$
This follows from the proof of \cite[Lemma 8]{grossparson}.
\end{proof}
We refer the reader to \S\ref{sec_Heegner_local_properties_at_p} for an analysis of the $p$-local properties of the Heegner classes.

\subsection{Shimura sets}
\label{subsec_Shimura_sets}
Fix an integer $$S\in \bigcup_n \fN_n^\de$$ and let $B_S$ denote the definite quaternion algebra over $\QQ$ of discriminant $N^-S$. 

Consider the zero-dimensional Shimura variety (``Shimura set'')  
$$X_S=X_{N^+,N^-S}:=B_S^\times \big\backslash \widehat{B}_S^\times\big/\widehat{R}_S^\times $$
attached to an Eichler order $R_S\subset B_S$ of level $N^+$, where
$\widehat{B}_S^\times:=B_S\otimes_{\ZZ}\widehat{\ZZ}$ and $\widehat{R}_S^\times:=R_S\otimes_{\ZZ}\widehat{\ZZ}$ as before.

As in \S\ref{subsubsec_522_2023_04_07}, define $x_m^0 \in X_S(K[p^m])$ as the image of
$$
\left[\begin{pmatrix}
    p^m & 0\\
0& 1
\end{pmatrix}\right]\in K^\times \big\backslash \widehat{K}^\times{\rm GL}_2 (\QQ_p)\widehat{B}_S^{(p)}\big/\widehat{R}_S^\times \subset K^\times \backslash \widehat{B}_S^\times/\widehat{R}_S^\times=X_S$$
and the degree-zero divisor
$$x_m:=(a_{q_0}(f)-q_0-1)^{-1}{\rm Tr}_{K[p^{d(m)}]/K_m}\,(T_{q_0}-q_0-1)[x_{m+1}^0]\in \ZZ_p[X_S]^0\,,$$
where $d(m)=\min\left\{d\in\ZZ_{\ge0}: K_m\subset K[p^{d}]\right\}$.

\subsubsection{Arithmetic level raising (bis)} 
\label{subsubsec_2024_10_04_1617}The Hecke algebra $\TT_S$  generated by $T_\ell$ for $\ell\nmid NS$ and $U_\ell$ for $\ell \mid NS$ acts on $\ZZ_p[X_S]^0$ via Picard functoriality. 

As before, a weak eigenform $h$ on $X_S$ modulo $p^n$ of weight $2$ is a ring homomorphism $\mathbb{T}_S\xrightarrow{e_h} \ZZ/p^n\ZZ$. In light of \cite[Lemma 4.5.5]{sweeting}, the choice of the weak eigenform $h$ yields\footnote{Note that the error term $C$ therein equals zero under our running hypotheses.} a group homomorphism $\ZZ_p[X_S]^0\lra \ZZ/p^n\ZZ$, which we denote by $h$ as well. 

We have the following key  existence (cf.~\cite[\S4.3]{pollack-weston11}, \cite[\S4.2]{chidahsiehanticyclomainconjformodformscomposito}, or \cite[proof of Theorem~5.15]{BertoliniDarmon2005}).

\begin{proposition} Let $E$ be a semistable elliptic curve and $p$ a prime such that
the hypothesis \eqref{item_Loc} holds.  Then for $S \in \fN_n^\de$ there exists a weak eigenform 
\begin{align}
    \begin{aligned}
    \label{eqn_5_9_2023_09_01_1755}
    f_S: \ZZ_p[X_S]^0\lra \ZZ/p^n\ZZ 
    \end{aligned}
\end{align}
satisfying \eqref{eq:f_S}.
\end{proposition} Such an $f_S$ is said to a level-raising the eigenform $f\mod p^n$ to level $S$. 

Put
\begin{align}
    \label{eqn_2023_04_10_1013}
    \begin{aligned}
         \lambda(S,n)_m^{(\sigma)}&:=f_S(\sigma\cdot x_m)\in \ZZ/p^n\ZZ\,,\qquad \sigma\in G_m:=\Gal(K_m/K)\,,\\
\lambda(S,n)_m&:=\sum_{\sigma\in G_m} \lambda(S,n)_m^{(\sigma)}\,\,\sigma\ \in \ \ZZ/p^n\ZZ[G_m]\,.
    \end{aligned}
\end{align}

\subsubsection{Norm relations (bis)} \label{S:norm-relations}
Let $ \pi_{m+1,m}: \ZZ/p^n\ZZ[G_{m+1}]\to  \ZZ/p^n\ZZ[G_m]$ denote the natural map induced by the projection on Galois groups, and $\xi_m:  \ZZ/p^n\ZZ[G_m] \to \ZZ/p^n\ZZ[G_{m+1}]$ denote the norm map given by $$G_m\ni \sigma\mapsto {\displaystyle\sum_{\substack{\widetilde{\sigma}\in G_{m+1} \\ \pi_{m+1,m}(\widetilde{\sigma})=\sigma}}\widetilde{\sigma}}.$$

The formal trace compatibility relations enjoyed by $\{x_m\}_m$ yields the following:
\begin{equation}
    \label{eqn_2023_04_17}
    \pi_{m+1,m}\lambda(S,n)_{m+1}=-\xi_{m-1}\lambda(S,n)_{m-1}\,,\qquad \forall\,m\in \ZZ^+\,.
\end{equation}

\begin{lemma}
\label{lemma_DI_Prop_2_8}
Fix positive integers $m$ and $n$. Let $\epsilon$ denote the sign of $(-1)^m$. We then have 
$$\omega_m^{\epsilon} \lambda(S,n)_{m}=0\,.$$
Moreover, there is a unique element 
$$L(S,n)_{m}^{\epsilon}\in \LL_{n}/(\omega_m)$$
such that $\lambda(S,n)_{m}=\widetilde{\omega}_m^{-\epsilon}L(S,n)_{m}^{\epsilon}$\,.
\end{lemma}

\begin{proof}
    This follows from \eqref{eqn_2023_04_17}, cf. \cite[Proposition 2.8]{darmoniovita}.
\end{proof}
For $L(S,n)_{m}^{\epsilon}$ as in the statement of Lemma~\ref{lemma_DI_Prop_2_8}, put
\begin{align}
        \label{defn_2023_04_17_1545}
        \begin{aligned}
            \lambda(S,n)_{m}^{+}&:=(-1)^{\frac{m}{2}}L(S,n)_{m}^{+},\qquad \hbox{if $m$ is even};\\
\lambda(S,n)_{m}^{-}&:=(-1)^{\frac{m+1}{2}}L(S,n)_{m}^{-},\qquad \hbox{if $m$ is odd\,,}
        \end{aligned}
\end{align}
to ensure the following compatibility:

\begin{lemma}
   \label{lemma_DI_Lemma_2_9}
 The sequence $\{\lambda(S,n)_{m}^{\pm}:\, {(-1)^m=\pm 1}\}$ is compatible with respect to natural projections 
 $$\LL_n/(\omega_{m+2}^\pm)\lra \LL_n/(\omega_{m}^\pm)\,, \qquad m\in \ZZ^+\,.$$
\end{lemma}

\begin{proof}
See \cite[Lemma 2.9]{darmoniovita}.
\end{proof}

In view of Lemma~\ref{lemma_DI_Lemma_2_9}, we define 
\begin{equation}
\label{eqn_2023_04_17_1607}
\lambda(S,n)^{\pm}:=\{\lambda(S,n)_{m}^{\pm}\,: (-1)^m=\pm 1\}\in \varprojlim_{(-1)^m=\pm 1} \LL_n/(\omega_m^\pm)\simeq \LL_n\,.
\end{equation}

\subsection{Level raising and rank lowering}
The following key proposition strengthens Zhang's principle of mod $p$ level raising and rank lowering in \cite{zhang14}. 

\begin{proposition}
    \label{prop_sweeting_lifting}
    In the setting of \S\ref{subsubsec_2024_10_04_1617}, 
    there exist infinitely many $S\in \fN_n^\de$ and newforms $g_{S} \in S_{2}(\Gamma_{0}(N_{0}S))$  with the following properties. 
    \begin{itemize}
    \item[a)] The Hecke field of $g_S$ is contained in $\QQ_p$.
    \item[b)] The form $g_S$ is twisted-Steinberg at $\ell\mid S$ with character\footnote{Recall that $\epsilon_\ell:\QQ_\ell^\times \to \{\pm 1\}$ is the unique character given by $\ell \mapsto \epsilon_\ell(f)$.} $\epsilon_\ell$.  Moreover, the local Galois representations ${\rho_f}_{\vert_{I_\ell}}$ and ${\rho_{g_S}}_{\vert_{I_\ell}}$ are isomorphic for all $\ell\nmid S$. 
    \item[c)] The $G_\QQ$-representations $T_{f,n}$ and $T_{g_S,n}$ are isomorphic. 
    \item[d)] The Bloch--Kato Selmer group ${\rm Sel}(K,A_{g_S})$ is finite.\\
    \end{itemize}
 Moreover, if $N_0$ is square-free and 
 the hypothesis \eqref{item_Loc} holds, then there exist infinitely many $S\in \fN_1^{\rm def}$ and newforms 
 $g_S \in S_2(\Gamma_0(N_0S))$ satisfying c) and d) above such that ${\rm Sel}(K,A_{g_S})=\{0\}$.
\end{proposition}

\begin{proof}
The existence of a newform $g_S$ satisfying (a)-(d) follows from the work of Fakhruddin--Khare--Patrikis \cite{FKP}, as shown in \cite[Theorem 6.2.4]{sweeting}
in combination with the vanishing $$d_{\texttt{Q}}=r_{\texttt{Q}}=0$$ of Proposition 6.3.5 in {\it loc. cit.} 
      
The final assertion is shown by Zhang in the course of the proof of \cite[Theorem~9.1]{zhang14}. Indeed, note that the mod $p$ Galois representation associated with the newform $f=f_{E}$ is surjective since $p\geq 5$ is a supersingular prime for the semistable elliptic curve $E$ (cf.~\cite[Proposition~2.1]{edi97}) and so the first three hypotheses of \cite[Theorem~9.1]{zhang14} are satisfied, which suffice for his method of mod $p$ level raising and rank lowering.
\end{proof}
\begin{remark}\label{rmk:lr}
For our main results, 
the `Moreover' part of Proposition~\ref{prop_sweeting_lifting}  suffices: 
Theorem~\ref{thm:main} assumes the hypothesis \eqref{item_Loc} and $N_0$ to be square-free.
With an eye towards eventual removal of these hypotheses, we include the above more general setting. In particular, the properties a) and b) are included due to their relevance to arguments in \cite{sweeting}. 
\end{remark}
\subsection{Interpolation formulas}
Throughout this subsection, fix a positive integer $n$.
\subsubsection{} For a positive integer $n$, fix $S\in \fN_n^\de$ and a $p$-primitive newform $$g=g_S\in S_2(\Gamma_0(N_0S))$$ as in  Proposition~\ref{prop_sweeting_lifting}.

This subsection concerns the values of the elements $\lambda(S,n)^\pm$ as in \eqref{eqn_2023_04_17_1607} at the trivial character $\one$ of $\Gamma$. The main result is Lemma~\ref{lem:interpolation} below.

For $m\ge0$, define $\lambda(g)_m\in \Zp[G_m]$ by 
\[
\lambda(g)_m = \sum_{\sigma\in G_m}g^{\JL}(\sigma\cdot x_m)\sigma,
\]
where $g^\JL$ is a $p$-primitive\footnote{Namely, $g^\JL$ is non-zero modulo $p$.} Jacquet--Langlands transfer of $g$ to $X_S$.  The latter exists since $S$ consists of admissible primes. 

Let $\alpha$ and $\beta$ be the two roots of the Hecke polynomial of $g$ at $p$. Denote by $\lambda(g)_{\alpha,m}$ and $\lambda(g)_{\beta,m}$ the 
theta elements attached to the $p$-stabilizations of $g$ with respect to these two roots, given as in \cite[Definition~4.1]{chidahseihanticycloLvaluescrelle}. More precisely, we have the formula
\[
\lambda(g)_{\clubsuit,m}=\frac{1}{\clubsuit^m}\lambda(g )_m-\frac{1}{\clubsuit^{m+1}}\xi_{m-1}\lambda(g )_{m-1}, \quad m\ge 1,\ \clubsuit\in\{\alpha,\beta\}.
\]
Here, $\xi_{m-1}$ is the norm map defined at the beginning of \S\ref{S:norm-relations}. For $m=0$, we have
\[
\lambda(g)_{\clubsuit,0}=\lambda(g )_0-\frac{1}{\clubsuit}\left(a_p\lambda(g)_0-\pi_{1,0}(\lambda(g )_1)
\right),\quad \clubsuit\in\{\alpha,\beta\}.
\]


As we are in the non-ordinary case, these elements are no longer integral. Instead, they are elements of $\Qp(\clubsuit)[G_m]$.
Lemma 4.2 of {\it op. cit.} tells us that these elements are compatible under $\pi_{m+1,m}$. In particular, 
\[\lambda(g)_{\clubsuit,1}(\one)=\lambda(g)_{\clubsuit,0}(\one)\]
for $\clubsuit\in\{\alpha,\beta\}$.
Furthermore, it follows from Theorem~4.6 of {\it op. } that
\begin{align}
\label{eqn_2023_09_06_1914}
\lambda(g)_{\clubsuit,0}(\one)^2=\lambda(g)_{\clubsuit,0}(\one)^2=u(1-\clubsuit^{-2})\cdot \frac{L(g/K,1)}{\Omega_g^{\mathrm{can}} \eta_{g,N^+,N^-S}}
\end{align}
for some $p$-adic unit $u$, where $\Omega_g^{\mathrm{can}} \eta_{g,N^+,N^-S}$ is the Gross period, given as in \cite[\S6.2]{zhang14}.

\begin{lemma}\label{lem:interpolation}
   We have $\lambda(S,n)^+(\one)=0$, whereas  $\lambda(S,n)^-(\one)^2$ equals 
   $$\frac{L(g/K,1)}{\Omega_g^{\mathrm{can}} \eta_{g,N^+,N^-S}} \mod p^n\,$$
  up to a $p$-adic unit.     
\end{lemma}
\begin{proof}
    Since $\omega_m^\pm$ vanishes at $\one$, the equality 
    \[\lambda(S,n)^\pm(\one)=\lambda(S,n)_m^\pm(\one)\]
    holds for any $m\ge0$ such that $(-1)^m=\pm1$. In particular, 
    \begin{align*}
        \lambda(S,n)^+(\one)&=\lambda(S,n)_0^+(\one)=f_S( x_0),\\
        \lambda(S,n)^-(\one)&=\lambda(S,n)_1^-(\one)=-\lambda(S,n)_1(\one).
    \end{align*}
    As explained in \cite[\S2.4]{BD-mumford},  $x_0$ is a multiple of $T_p(x_0^0)$, where $x_0^0$ is the Heegner point of trivial conductor introduced in \S\ref{subsubsec_522_2023_04_07} and $T_p$ is the Hecke operator. Since $a_p(f)=0$, it follows that $f_S(x_0)=0$. In particular, $\lambda(S,n)^+(\one)=0$.

We next turn our attention to the computation of $\lambda(S,n)^-(\one)$. Recall that the newform $g$ is congruent to $f_S$ mod $p^n$. By definition, 
\begin{align*}
 \alpha^2\lambda(g)_{\alpha,1}-\beta^2\lambda(g)_{\beta,1}&=(\alpha-\beta)\lambda(g)_1,\\
\alpha\lambda(g)_{\alpha,0}-\beta\lambda(g)_{\beta,0}&=(\alpha-\beta)\lambda(g)_0.   
\end{align*}
Put 
$$\fu:=u\cdot \frac{L(g/K,1)}{\Omega_g^{\mathrm{can}} \eta_{g,N^+,N^-S}}\,,\qquad \fv_\clubsuit:=\lambda(g)_{\clubsuit,1}(\one)\,.$$ 
Then,
\begin{align*}
   (\alpha-\beta)^2 \lambda(g)_1(\one)^2&=(\alpha^4-\alpha^2+\beta^4-\beta^2)\fu-2\alpha^2\beta^2\fv_\alpha\fv_\beta\\
      (\alpha-\beta)^2 \lambda(g)_0(\one)^2&=(\alpha^2+\beta^2-2)\fu-2\alpha\beta\fv_\alpha\fv_\beta.
\end{align*}
Combining the two equations to cancel the term $\fv_\alpha\fv_\beta$ gives
\[
\lambda(g)_1(\one)^2-\alpha\beta\lambda(g)_0(\one)^2=\left(\frac{\alpha^3-\beta^3}{\alpha-\beta}-1\right)\fu.
\]
Thanks to our choice of $g$, it follows that 
$$\lambda(g)_0(\one)=g^{\rm JL}(x_0)\equiv f_S(x_0)=0 \mod p^n\,. $$
Here, all but the final equality follows from definitions, and the final equality is verified in the first paragraph of this proof.
Since the factor $\left(\frac{\alpha^3-\beta^3}{\alpha-\beta}-1\right)$ is a $p$-adic unit, the assertion on $\lambda(S,n)^-(\one)$ follows.
\end{proof}

\section{Signed bipartite Euler systems}
\label{sec_Heeg_classes_construction_reciprocity}
 The aim of this section is to introduce plus/minus Heegner point bipartite Euler systems in our inert supersingular setting based on the work of Howard \cite{howard06} and verify an analogue of his equality criterion, leading to the proof of Theorem~\ref{thm:main}. 
Another salient feature of this section 
 is the verification of $p$-local properties of these plus/minus classes
(cf.~\S\ref{sec_Heegner_local_properties_at_p}).

\subsection{General framework}
\label{S:generaltheory} 
We keep the notation of \S\ref{S:Shimura}. 

\begin{defn}
\label{defn_6_1_2023_04_17}
    For $\bullet\in\{+,-\}$,  a $\bullet$-bipartite Euler system for the $G_\QQ$-representation $T_f$ consists of two families of elements
    \[
    \{\hbox{\upshape\texttt K}(S,n)^\bullet\in \widehat{H}^{1,\bullet}_{{\rm ord}_S}(K_\infty,\Tfn):S\in\fN_n^\inde\},\quad \{\hbox{\upshape\texttt L}(S,n)^\bullet\in \Lambda_n:S\in\fN_n^\de\}
    \]
    for all positive integers $n$ with the following properties. 
    \begin{itemize}
        \item[i)]The elements $\hbox{\upshape\texttt K}(S,n)^\bullet$ and $\hbox{\upshape\texttt L}(S,n)^\bullet$ are compatible under the natural maps as $n$ varies.
        \item[ii)]For any $S\ell\in \fN_n^\inde$, there is an isomorphism of $\Lambda$-modules
        \[
        \varprojlim_m H^1_\ord (K_{m,\ell},\Tfn)\simeq \Lambda_n
        \]
        sending $\res_\ell(\hbox{\upshape\texttt K}(S\ell,n)^\bullet)$ to $\hbox{\upshape\texttt L}(S,n)^\bullet$.
        \item[iii)]For any $S\ell\in \fN_n^\de$, there is an isomorphism of $\Lambda$-modules
        \[
        \varprojlim_m H^1_\ur(K_{m,\ell},\Tfn)\simeq \Lambda_n
        \]
        sending $\res_\ell(\hbox{\upshape\texttt K}(S,n)^\bullet)$ to $\hbox{\upshape\texttt L}(S\ell,n)^\bullet$.
    \end{itemize}
\end{defn}

Note that when $S=1$ is the empty product, we in particular have
$$\hbox{\upshape\texttt K}(1,n)^\bullet \in \widehat{H}^{1,\bullet}(K_\infty,\Tfn)$$
for all positive integers $n$, and 
$$\hbox{\upshape\texttt K}(1)^\bullet:=\{\hbox{\upshape\texttt K}(1,n)^\bullet\}_n \in \varprojlim_n\widehat{H}^{1,\bullet}(K_\infty,\Tfn)=H^{1,\bullet}(K,\TT_{f}^\ac)\,$$
for $\TT_{f}^\ac:=T_f \otimes \Lambda$. 

The following links the plus/minus bipartite Euler systems with anticyclotomic Iwasawa theory. 
\begin{theorem}
\label{thm_bipartite_abstract_main}
For $\bullet\in \{+,-\}$, let $(\{\hbox{\upshape\texttt K}(S,n)^\bullet\},\{\hbox{\upshape\texttt L}(S,n)^\bullet\})$ be a $\bullet$-bipartite Euler system. Assume that $\hbox{\upshape\texttt K}(1)^\bullet$ is non-zero. 
Then the following holds.
\begin{itemize}
    \item[i)] We have $\rank_\Lambda\, H^{1,\bullet}(K,\TT_f^\ac)=\rank_\Lambda\,\Sel_\bullet(K_\infty,A_f)^\vee=1$. 
    \item[ii)] 
    We have
    $${\rm char}_{\LL}\left(H^{1,\bullet}(K,\TT_{f}^\ac)/\LL\cdot\hbox{\upshape\texttt K}(1)^\bullet \right)^2 \subset {\rm char}_{\LL}(\Sel_\bullet(K_\infty,A_f)^\vee_{\rm tor})\,.$$
    \item[iii)] The inclusion in part (ii) is an equality upon localization at a height one prime $\fP$ of $\LL$ under the following hypothesis: There exists a positive integer $k_0=k_0(\mathfrak{P})$ with the property that for an integer $n\geq k_0$, there exists an integer $S\in \fN_n^{\de}$ such that $\hbox{\upshape\texttt L}(S,n)^\bullet \in \LL_n$ has non-zero image in $\LL/(\mathfrak{P},p^n)$. In particular, if the latter is satisfied by all height one prime ideals $\fP$ of $\LL$, then the inclusion in part (ii) is an equality. 
\end{itemize}
\end{theorem}

Our proof of Theorem~\ref{thm_bipartite_abstract_main} is based on Howard's strategy via which he proved \cite[Theorem 3.2.3]{howard06}. We will describe the key steps in  \S\ref{subsec_proof_thm_bipartite_abstract_main}, highlighting
the differences. 

\subsection{Proof of Theorem~\ref{thm_bipartite_abstract_main}}
\label{subsec_proof_thm_bipartite_abstract_main}
\subsubsection{}
\label{subsubsec_2023_04_28_1240} We first introduce some notation.

Fix a height-one prime $\fP\neq(p)$ of $\LL$. Let  $\cO_\fP$ denote the integral closure of $\LL/\fP$ and put $\Phi_\fP:={\rm Frac}(\cO_\fP)$, which is a finite extension of $\QQ_p$. Recall that the modular parametrization of $E$ gives rise to an isomorphism $T_f\simeq T_p(E)$. 
To ease the notation, put $T=T_p(E)$ and $\TT=T\otimes \LL$, on which $G_K$ acts naturally.

Put $T_\fP:=T\otimes_{\ZZ_p}\cO_\fP$, which we endow with diagonal $G_K$-action, where the action on $\cO_\fP$ is given by the universal anticyclotomic character $$G_K\rightarrow \Lambda^\times \rightarrow (\LL/\fP)^\times \hookrightarrow \cO_\fP^\times.$$ Define $V_\fP:=T_\fP[\frac{1}{p}]=T\otimes_{\Zp} \Phi_\fP$, and $W_\fP:=V_{\fP}/T_{\fP}\simeq E[p^\infty]\otimes_{\ZZ_p}\cO_\fP$. We have a short exact sequence
\begin{equation}
    \label{eqn_3term_TVW_fP}
    0\lra T_\fP \lra V_\fP \lra W_\fP \lra 0\,.
\end{equation}
\subsubsection{} For $\bullet \in \{+,-\}$, we define the local conditions $\FFF_\fP^\bullet$ on $V_\fP$ as follows. 

If $q \nmid pN^-$ is a prime ideal of $\cO_K$, then
$H^1_{\cF_\fP^\bullet}(K_q,V_\fP)$ is the unramified classes. 

If $q\mid N^-$, then $H^1_{\cF_\fP^\bullet}(K_q,V_\fP)$ is generated by the image of the Kummer map:
$$H^1_{\cF_\fP^\bullet}(K_q,V_\fP)={\rm im}\left(E(K_q)\otimes_{\ZZ} \Phi_\fP \lra H^1(K_q,T_p(E))\otimes_{\ZZ_p}\Phi_\fP \lra H^1(K_q,V_\fP) \right)\,, \qquad q\mid N^-\,.$$

Let $w_1,\dots ,w_r$ be the primes of $K_\infty$ lying above $p$. Suppose that $\{\sigma_1,\dots,\sigma_r\}$ is a set of coset representatives of $\Gamma/\Gamma_{w_1}$ such that $w_i=\sigma_i\cdot w_1$ for $i=1,\dots,r$. Consider the following semi-local Coleman maps
\begin{align*}
\col^\bullet:H^1(K_p,\TT)=\bigoplus_{i=1}^rH^1(K_p,T\otimes\Zp[[\Gamma_{w_1}]])\cdot\sigma_i&\lra \Lambda\otimes\cO_{K_p}\\
    (z_i\sigma_i)_{1\le i\le r}&\longmapsto\sum_{i=1}^r\col^\bullet(z_i)\cdot \sigma_i,
\end{align*}
where $\col^\bullet:H^1(K_p,T\otimes\Zp[[\Gamma_{w_1}]])\rightarrow \Zp[[\Gamma_{w_1}]]\otimes\cO_{K_p}$ is the map given by $\col^\bullet_{T_f}$ defined in \S\ref{sec:coleman}.

Define $H^1_{\cF_\fP^\bullet}(K_p,V_\fP)$ as the $\Phi_\fP$-vector space generated by the image of 
$$H^{1,\bullet}(K_p,\TT):=\ker\left(H^1(K_p,\TT)\xrightarrow{{\rm Col}^\bullet} \LL\otimes \cO_{K_p}\right)$$
under the chain of natural morphisms 
$$H^1(K_p,\TT)\lra H^1(K_p,T\otimes \LL/\fP)\lra H^1(K_p,T_\fP)\lra H^1(K_p,V_\fP)\,. $$

\subsubsection{} Denote the local conditions on $T_\fP$ and $W_{\fP}$ that are propagated from $\FFF_\fP^\bullet$ on $V_\fP$ via \eqref{eqn_3term_TVW_fP}, in the sense of \cite[Example~1.1.2]{mr02}, also by $\FFF_\fP^\bullet$.

\subsubsection{}
\label{subsubsec_624_2023_04_19_1625} The above local conditions coincide with the ones in \cite[\S3.3]{howard06}, except for the case $q=p$. Accordinlgy, for primes $q\neq p$, the proof of Theorem~3.2.3 of {\it op. cit.} applies verbatim. 

In particular, the corresponding local conditions $\FFF_\fP^\bullet$ on $T_\fP$ are cartesian in the sense of \cite[Definition 1.1.4]{mr02} and \cite[Definition 2.2.2]{howard06} by  \cite[Lemma 3.3.4]{howard06}.

\subsubsection{} For $q=p$, the local conditions determined by $\FFF_\fP^\bullet$ on the quotients of $T_\fP$ are also cartesian as explained below. 

In view of \cite[Lemma 3.7.1]{mr02}, it suffices to check that the quotient $H^1(K_p,T_\fP)/H^1_{\FFF_\fP^\bullet}(K_p,T_\fP)$ is $\cO_\fP$-torsion free. This is clear, since 
$$H^1_{\FFF_\fP^\bullet}(K_p,T_\fP):=\ker\left(H^1(K_p,T_\fP)\lra \frac{H^1(K_p,V_\fP)}{H^1_{\FFF_\fP^\bullet}(K_p,V_\fP)} \right)$$
by the definition of the propagated local conditions $\FFF_\fP^\bullet$ on $T_\fP$.
\subsubsection{}
We further analyze the local conditions $\FFF_\fP^\bullet$ at $p$, which we will use to prove a local control theorem, namely  Proposition~\ref{prop_control_thm_at_p_fP} below.

\begin{lemma}
\label{lemma_2023_04_18_2107}
The natural map 
$H^1(K_p,\TT)\to H^1(K_p,T\otimes \LL/\fP)$ is surjective.
\end{lemma}
\begin{proof}
Note that $$H^2(K_p,E[p])\simeq H^0(K_p,E[p])^\vee=\{0\},$$ where the latter holds since $p$ is a prime of supersingular reduction. Moreover, since the natural map $H^2(K_p,\TT)\to H^2(K_p,E[p])$ 
is surjective, it follows from Nakayama's lemma that $H^2(K_p,\TT)=\{0\}$. The proof concludes since 
$${\rm coker}\,(H^1(K_p,\TT)\lra H^1(K_p,T\otimes \LL/\fP))=H^2(K_p,\TT)[\fP].$$
\end{proof}

In view of Lemma~\ref{lemma_2023_04_18_2107}, we may define the map
\begin{equation}
\label{eqn_2023_04_18_2119}
\col^\bullet_{/\fP}\,:\quad H^1(K_p,T\otimes \LL/\fP)\lra \LL/\fP \otimes \cO_{K_p}
\end{equation}
 for $\bullet\in\{+,-\}$, 
via the following commutative diagram:
\begin{equation}
\begin{aligned}
\label{eqn_2023_04_18_2120}
    \xymatrix{
    H^1(K_p,\TT)\ar[rr]^(.58){\col^\bullet}\ar@{->>}[d]&& \LL\otimes\cO_{K_p}\ar@{->>}[d]\\
H^1(K_p,T\otimes \LL/\fP)\ar[rr]_(.58){\col^\bullet_{/\fP}}&& \LL/\fP \otimes \cO_{K_p}.
}
\end{aligned}
\end{equation}

Let us define $H^{1,\bullet}(K_p,T\otimes \LL/\fP)$ as the image of $H^{1,\bullet}(K_p,\TT)$. Moreover, since the map $\col^\bullet$ is surjective (cf. \S\ref{sec:coleman}), it follows from \eqref{eqn_2023_04_18_2120} that the map $\col^\bullet_{/\fP}$ is surjective as well.

The following diagram with cartesian squares and exact rows
\begin{equation}
\begin{aligned}
\label{eqn_2023_04_18_2127}
    \xymatrix{
 0\ar[r]&H^{1,\bullet}(K_p,\TT)\ar@{-->}[d]_{a_\bullet} \ar[r] & H^1(K_p,\TT)\ar[r]^(.58){\col^\bullet}\ar@{->>}[d]_{a}& \LL\otimes\cO_{K_p}\ar@{->>}[d]^b\ar[r] &0\\
 0\ar[r]& \ker(\col^\bullet_{/\fP})\ar[r]&H^1(K_p,T\otimes \LL/\fP)\ar[r]_(.58){\col^\bullet_{/\fP}}& \LL/\fP \otimes \cO_{K_p}\ar[r]&0
}
\end{aligned}
\end{equation}
shows that 
\begin{equation}
    \label{eqn_2023_04_18_2132}
    H^{1,\bullet}(K_p,T\otimes \LL/\fP)={\rm im}(a_\bullet)\,\subset\,  \ker(\col^\bullet_{/\fP})\,.
\end{equation}
Moreover, observe that $\ker(a)=\fP H^1(K_p,\TT)$ maps surjectively onto $\ker(b)=\fP \otimes_{\Zp} \cO_{K_p}$. As a result, it follows from the snake lemma applied with the diagram \eqref{eqn_2023_04_18_2127} that the map $a_\bullet$ is surjective. Combining this fact with \eqref{eqn_2023_04_18_2132}, we conclude that
\begin{equation}
    \label{eqn_2023_04_18_2139}
    H^{1,\bullet}(K_p,T\otimes \LL/\fP)=\ker\left(H^1(K_p,T\otimes \LL/\fP)\xrightarrow{\col^\bullet_{/\fP}} \LL/\fP\otimes_{\Zp} \cO_{K_p}\right)\,.
\end{equation}

Applying the functor $-\otimes_{\LL/\fP}\cO_\fP$ to \eqref{eqn_2023_04_18_2119}, we have a surjective morphism
$$\col^\bullet_{\fP}\,:\quad H^1(K_p,T_\fP)\lra \cO_\fP \otimes_{\Zp} \cO_{K_p}\,.$$

\begin{lemma}
    \label{lemma_2023_04_18_2148} 
    We have
    $$H^1_{\FFF_\fP^\bullet}(K_p,T_\fP)=\ker\left(\col^\bullet_{\fP}\right)\simeq  H^{1,\bullet}(K_p,T\otimes \LL/\fP)\otimes_{\LL/\fP}\cO_\fP\,.$$
\end{lemma}
\begin{proof}
    The first identification follows from \eqref{eqn_2023_04_18_2139} and the definition of $H^1_{\FFF_\fP^\bullet}(K_p,T_\fP)$. The final isomorphism follows from the fact that the map $\col^\bullet_{/\fP}$ is surjective and ${\rm Tor}_1^{\LL/\fP}(\LL/\fP,\cO_\fP)=0$.
\end{proof}

\subsubsection{} Let $\iota:\Gamma\to\Gamma$ denote the involution $\gamma\mapsto \gamma^{-1}$. As explained in \cite{howard04}, we have a perfect $G_K$-invariant pairing 
$$T_{\fP^\iota}\times W_{\fP}\lra \mu_{p^\infty}\,.$$
Passing to the dual of the natural map $T\otimes \LL/\fP^\iota \hookrightarrow T_{\fP^\iota}$ via the perfect pairing above, we have a $G_K$-equivariant map
\begin{equation}
\label{eqn_2023_04_19_1601}
    W_\fP\lra \left(\varinjlim_n \Ind_{K_n/K} E[p^\infty]\right)[\fP]
\end{equation}
of $\LL$-modules. 
\subsubsection{} 
The following control theorem for signed local conditions at $p$ is a technical ingredient in our later arguments. 
\begin{proposition}
    \label{prop_control_thm_at_p_fP}
Let $\fP\neq p\Lambda$ be a height-one prime and $\bullet\in\{+,-\}$. Let  $k_\infty$ be the anticyclotomic $\ZZ_p$-extension of the unramified quadratic extension $K_p$ of $\QQ_p$.
Then the maps \[
H^{1,\bullet}(K_p,T\otimes\Lambda/\fP)\lra H^1_{\cF^\bullet_\fP}(K_p,T_\fP),
\]
\[
H^1_{\cF^\bullet_\fP}(K_p,W_\fP)\lra H^1_\bullet(k_\infty,E[p^\infty])[\fP],
\]
where the latter is induced from \eqref{eqn_2023_04_19_1601}, have finite kernels and cokernels bounded by constants that depend only on $[\cO_\fP:\LL/\fP]$.

\end{proposition}
\begin{proof}
The first claim is clear thanks to Lemma~\ref{lemma_2023_04_18_2148}. The second claim\footnote{This argument is based on the final paragraph of \cite[proof of Lemma~2.2.7]{howard04}.} follows from the first claim applied to $\fP^\iota$ instead $\fP$ and duality. 
\end{proof}

\subsubsection{} \label{S:controls}
We describe further preliminaries towards our proof of Theorem~\ref{thm_bipartite_abstract_main}, which are based on  \cite[Proof~of~Theorem~3.2.3]{howard06}. Although the contents of \S\S\ref{S:controls}--\ref{S:conclude} closely follow {\it loc. cit.} we include details for convenience of the reader, highlighting peculiarities of our setting. 

For $X=T_\fP, W_\fP$, or their sub-quotients on which the local conditions are prescribed by propagation, denote by $H^1_{\cF_\fP^\bullet}(K,X)$ the Selmer group determined by the local conditions $\cF_\fP^\bullet$.

\begin{proposition}
    \label{prop_Howard_Prop_331}
    The natural maps
    $$H^{1,\bullet}(K,\TT)/\fP H^{1,\bullet}(K,\TT_{f}^\ac) \lra H^1_{\cF_\fP^\bullet}(K,T_\fP)\qquad\qquad H^1_{\cF_\fP^\bullet}(K,W_\fP)\lra {\rm  Sel}_\bullet (K_\infty/K, E[p^\infty])[\fP]$$
    satisfy the following properties. 
    \begin{itemize}
    \item[i)] The first map is injective.
    \item[ii)] There is a finite set of height-one primes $\Sigma_\LL$ of $\LL$ such that if $\fP\not\in \Sigma_\LL$, then these maps have finite kernel and cokernel which are bounded by a constant depending only on $[\cO_\fP:\LL/\fP]$.
    \end{itemize}
\end{proposition}
\begin{proof}
The proof of \cite[Proposition 3.3.1]{howard06} applies verbatim, where we resort to Proposition~\ref{prop_control_thm_at_p_fP} instead of \cite[Lemma 2.2.7]{howard04}. See also \S\ref{subsubsec_624_2023_04_19_1625} for a comparison of local conditions at primes other than $p$.  
\end{proof}

\begin{lemma}
    \label{lemma_howard_06_Lemma332}
    The natural map 
    $$H^1_{\cF_\fP^\bullet}(K,T_\fP)/p^kH^1_{\cF_\fP^\bullet}(K,T_{\fP})\lra H^1_{\cF_\fP^\bullet}(K,T_\fP/p^kT_\fP)$$
    is injective, where the local conditions $\cF_\fP^\bullet$ on $T_\fP/p^kT_\fP$ are obtained from those on $T_\fP$ by propagation.
\end{lemma}

\begin{proof}
The proof of \cite[Lemma 3.3.2]{howard06}, which in turn is a special case of \cite[Lemma 3.7.1]{mr02}, applies.
\end{proof}

For any $j\geq k$, write $T_j:=T_\fP/p^jT_\fP$ and $R_j:=\cO_\fP/p^j\cO_\fP$ to ease notation. Denote by $\cF$ the local conditions on $T_j$ obtained by the propagation of $\cF_\fP^\bullet$ on $T_\fP$, as well as the local conditions on $W_\fP[p^j]$ propagated from $\cF_\fP^\bullet$ on $W_\fP$. A choice of a generator of the maximal ideal of $\cO_\fP$ determines an isomorphism $T_j\simeq W_\fP[p^j]$. Under such an isomorphism, the local conditions on $T_j$ and $ W_\fP[p^j]$ are identified, thanks to \cite[Lemma 3.7.1]{mr02}. As a result, we have 
\begin{equation}
    \label{eqn_2023_04_19_1717}
    H^1_\cF(K,T_j)\simeq  H^1_\cF(K,W_\fP[p^j])\simeq H^1_{\cF_\fP^\bullet}(K,W_\fP)[p^j],
\end{equation}
where the final isomorphism follows from \cite[Lemma 3.5.4]{mr02}.

\begin{lemma}
    \label{lemma_howard_06_Lemma334}
The triple $(T_j,\cF,\fL_j)$ satisfies Hypotheses 2.2.4 and 2.3.1, as well as the hypotheses of \S2.6 in \cite{howard06}.
\end{lemma}

\begin{proof}
In view of the proof of {\cite[Lemma 3.3.4]{howard06}}, it suffices to show that 
    the local conditions $\cF=\cF_\fP^\bullet$ are
    \begin{itemize}
        \item  cartesian in the sense of \cite[Definition 1.1.4]{mr02}\,;
        \item self-dual in the sense of \cite[\S2.6]{howard06}. 
    \end{itemize} 
    The arguments in the proof of \cite[Lemma 3.3.4]{howard06} apply to check these properties, which correspond to item (d) in {\it loc. cit.}
    More precisely,   the first follows from \cite[Lemma 3.7.1]{mr02}, since the local conditions are obtained via propagation from those on the $\Phi_\fP$-vector space $V_\fP$. The second follows from the self-duality of the local conditions on $V_\fP$ and the first property.
\end{proof}

\subsubsection{}  
Suppose that $S\in\fN_j^\inde$ and $\ell \mid S$ is an admissible prime. Since $\ell$ is inert in $K/\QQ$, 
we have a unique decomposition
\begin{equation}
\label{eqn_2023_04_24_1716}
T_{j}{}_{\vert_{G_{K_\ell}}}\simeq R_j(1) \oplus R_j\,.
\end{equation}

Define the local condition $\cF(S)$ on $T_j$ by modifying the local conditions at primes $\ell\mid S$ as follows:
$$H^1_{\cF(S)}(K_\ell,T_j):={\rm im}\left( H^1(K_\ell,R_j(1))\xrightarrow{\eqref{eqn_2023_04_24_1716}} H^1(K_\ell,T_j)\right)\,.$$
Denote by $H^1_{\cF(S)}(K,T_j)$ the associated Selmer group. Note that 
$${\rm im}\left( \widehat{H}^{1,\bullet}_{{\rm ord}_S}(K_\infty,\Tfn) \lra H^1(K,T_j)\right)\subset H^1_{\cF(S)}(K,T_j)\,.$$

\subsubsection{}
\label{subsubsec_2023_04_28_1558}
For each $S\in\fN_j^\inde$, denote by ${\overline{\hbox{\upshape\texttt K}}}_S^\bullet \in H^1_{\cF(S)}(K,T_j)$ the image of $\hbox{\upshape\texttt K}(S,n)^\bullet\in \widehat{H}^{1,\bullet}_{{\rm ord}_S}(K_\infty,\Tfn)$. Likewise, for any $S\in\fN_j^\de$, denote by ${\overline{\hbox{\upshape\texttt L}}}_S^\bullet$ the image of $\hbox{\upshape\texttt L}(S,j)^\bullet\in \LL/(p^j)$ under the natural map $\LL/(p^j)\to R_j$.

\begin{lemma}
\label{lemma_Howard06_lemmas335and336_combined}
\noindent\begin{itemize}
\item[i)] The collections
\begin{equation}
\label{eqn_2023_04_28_1440}
    \{\overline{\hbox{\upshape\texttt K}}_S^\bullet\,: \quad S\in \fN_j^\inde\}\qquad \{{\overline{\hbox{\upshape\texttt L}}}_S^\bullet\,: \quad S\in \fN_j^\de\}
\end{equation}
form a bipartite Euler system of odd type for $(T_j,\cF, \fL_j)$ in the sense of \cite[Definition 2.3.2]{howard06}, with $\fN^{\rm odd}:=\fN_j^\inde$ and $\fN^{\rm even}:=\fN^\de$.
\item[ii)] On shrinking $\fL_j$ if necessary, one may ensure that the bipartite Euler system \eqref{eqn_2023_04_28_1440} is free in the sense of \cite[Definition 2.3.6]{howard06}.
\end{itemize}
\end{lemma}

\begin{proof}
The proofs of \cite[Lemmas 3.3.5 and 3.3.6]{howard06} apply verbatim, thanks to Lemma~\ref{lemma_howard_06_Lemma334}.
\end{proof}

We will henceforth assume that the set of primes $\fL_j$ is shrunk so as to ensure that the bipartite Euler system \eqref{eqn_2023_04_28_1440} is free.

\subsubsection{}
\label{subsubsec_2023_04_28_1239}
For any pair of positive integers $k\leq j$, let us put
$$\delta_\fP(k,j):=\min\{ {\rm ind}(\hbox{\upshape\texttt L}(S,j)^\bullet, \cO_\fP/p^k\cO_\fP)\,:\,S\in\fN_j^\de\} \leq \infty\,,$$
where for any DVR $R$ with maximal ideal $\m$ and an $R$-module $M$, the index ${\rm ind}(m,M)$ of divisibility of an element $m$ in $M$ is the largest integer $n\leq \infty$ so that $m\in \m^n M$. Note that $\delta_\fP(k,j)$ increases monotonically in $j$, so we can define $\delta_\fP(k):=\lim_{j\to \infty} \delta_\fP(k,j)$.

\begin{proposition}
    \label{prop_howard06_prop333}
    Suppose that there exists a positive integer $k$ such that the image of $\hbox{\upshape\texttt K}^\bullet\in H^{1,\bullet}(K,\TT)$ in the quotient $H^1_{\cF_\fP^\bullet}(K,T_\fP)/(p^k)$ is non-zero. Then the following holds.
    \begin{itemize}
    \item[i)] The $\cO_\fP$-module $H^1_{\cF_\fP^\bullet}(K,T_\fP)$ is free of rank one.
    \item[ii)] The $\cO_\fP$-module $H^1_{\cF_\fP^\bullet}(K,W_\fP)^\vee$ is of rank one.
    \item[iii)] We have
    $${\rm length}_{\cO_\fP}\,\left(H^1_{\cF_\fP^\bullet}(K,W_\fP)^\vee_{{\rm tor}} \right)+2\delta_\fP(k) = 2\cdot {\rm length}_{\cO_\fP}\,\left(H^1_{\cF_\fP^\bullet}(K,T_\fP)\big{/}\cO_\fP\cdot{\hbox{\upshape\texttt K}}(1)^\bullet \right)\,,$$
    where the subscript ${\rm tor}$ means the maximal $\cO_\fP$-torsion submodule.
    \end{itemize}
\end{proposition}
\begin{proof}
We proceed as in the final paragraph of \cite[\S3.3]{howard06}. We briefly highlight the key steps. 

Let $k$ be as in the statement of the proposition. Let us define the class ${\overline{\hbox{\upshape\texttt K}}}_1^\bullet \in H^1_{\cF}(K,T_k)$ as in \S\ref{subsubsec_2023_04_28_1558} with $j=k$. It follows from \cite[Theorem 2.5.1]{howard06}, which applies thanks to Lemma~\ref{lemma_howard_06_Lemma332}, that 
$$H^1_{\cF_\fP^\bullet}(K,W_\fP)[p^k]\simeq H^1_{\cF}(K,T_k)\simeq R_k \oplus M_k \oplus M_k,$$
where the first isomorphism follows from \eqref{eqn_2023_04_19_1717}, with 
$${\rm length}_{\cO_\fP}(M_k)+\delta_\fP(k,j)={\rm ind}\left( {\overline{\hbox{\upshape\texttt K}}}_1^\bullet, H^1_{\cF}(K,T_k)\right) < k,\qquad  \forall j\geq 2k\,.$$
Here, the final inequality is valid since ${\overline{\hbox{\upshape\texttt K}}}_1^\bullet\neq 0$ by assumption. 

Since ${\rm length}_{\cO_\fP}(M_k)<k$ and $H^1_{\cF_\fP^\bullet}(K,W_\fP)$ is a cofinitely generated $\cO_\fP$-module, 
allowing $k$ to vary, it follows  that 
\begin{equation}
\label{eqn_2023_04_28_1644}
    H^1_{\cF_\fP^\bullet}(K,W_\fP)\simeq (\cO_\fP[1/p]\Big{/}\cO_\fP)\oplus M \oplus M
\end{equation}
for some $\cO_\fP$-module $M\simeq M_k$ of finite cardinality. We then conclude that
$$H^1_{\cF_\fP^\bullet}(K,T_\fP)\simeq\varprojlim_k H^1_{\cF_\fP^\bullet}(K,W_\fP)[p^k],$$
where the isomorphism follows from \eqref{eqn_2023_04_19_1717}, is a free $\cO_\fP$-module of rank one. This concludes the proofs of parts (i) and (ii).

For all $j\geq 2k$, it follows that
\begin{align}
    \label{eqn_align_2023_04_28_1643}
   \notag {\rm length}_{\cO_\fP}\left(H^1_{\cF_\fP^\bullet}(K,W_\fP)^{\vee}_{\rm tor} \right)+2\delta_{\fP}(k,j)\stackrel{\eqref{eqn_2023_04_28_1644}}{=} & {\rm length}_{\cO_\fP}\left(M\oplus M\right)+2\delta_{\fP}(k,j)\\
    =& \,2\cdot {\rm ind}\left( {\overline{\hbox{\upshape\texttt K}}}_1^\bullet, H^1_{\cF}(K,T_k)\right) \\
   \notag =&\, 2\cdot  {\rm length}_{\cO_\fP}\left(H^1_{\cF_\fP^\bullet}(K,T_\fP)\big{/}\cO_\fP\cdot {\hbox{\upshape\texttt K}}(1)^\bullet\right)\,.
\end{align}
Here, the equality \eqref{eqn_align_2023_04_28_1643} 
follows from \cite[Theorem 2.5.1]{howard06}.  The final equality above follows from part (i) and since ${\overline{\hbox{\upshape\texttt K}}}_1^\bullet$ belongs to the image of the injective map 
\begin{equation}
\label{eqn_2023_04_28_1711}
H^1_{\cF_\fP^\bullet}(K,T_\fP)/p^k H^1_{\cF_\fP^\bullet}(K,T_\fP)\lra H^1_{\cF}(K,T_k)
\end{equation}
by its very definition (cf. Lemma~\ref{lemma_howard_06_Lemma332}),  together with the observation that the image of \eqref{eqn_2023_04_28_1711} is an $\cO_\fP$-direct summand of $H^1_{\cF}(K,T_k)$ by the discussion above.

Letting $j\to \infty$, the proof concludes.
\end{proof}

\subsubsection{}\label{S:conclude} 
The first assertion in the statement of Theorem~\ref{thm_bipartite_abstract_main} can be proved by arguing exactly as in the first three paragraphs of \cite[\S3.4]{howard06}, and relying on Proposition~\ref{prop_Howard_Prop_331} and the first two parts of Proposition~\ref{prop_howard06_prop333}.

The second and third assertions of Theorem~\ref{thm_bipartite_abstract_main} can be restated as follows: 
\begin{itemize}
    \item[a)] For all height-one primes $\fP$ of $\LL$, we have
    $${\rm ord}_{\fP}\left({\rm char}_{\LL}(\Sel_\bullet(K_\infty,A_f)^\vee_{\rm tor})\right) \leq 2\cdot {\rm ord}_{\fP}\left(H^{1,\bullet}(K,\TT)/\LL\cdot\hbox{\upshape\texttt K}(1)^\bullet  \right)\,.$$
    \item[b)] We have equality in (a) if the following holds: There exists a positive integer $k_0=k_0(\mathfrak{P})$ with the property that for some integer $n\geq k_0$, there exists an integer $S\in \fN_n^{\de}$ such that $\hbox{\upshape\texttt L}(S,n)^\bullet \in \LL_n$ has non-zero image in $\LL/(\mathfrak{P},p^n)$.
\end{itemize}

For a height-one prime $\fP\neq p\LL$ generated by a distinguished polynomial $f$, one may  prove both assertions as in \cite[\S3.4]{howard06}, by applying Propositions~\ref{prop_Howard_Prop_331} and \ref{prop_howard06_prop333} with the collection of height one primes of the form $\fP_m:=(f+p^m)\LL$ for integers $m\gg0$. 

For $\fP=p\LL$, one still argues the same way, but now applying Propositions~\ref{prop_Howard_Prop_331} and \ref{prop_howard06_prop333}  with the collection of height one primes of the form $\fP_m:=((\gamma-1)^m+p)\LL$ for integers $m\gg0$, where $\gamma\in \Gamma$ is a topological generator.\qed

\subsection{Signed Heegner point bipartite Euler system}
\label{subsec_8_1_2022_09_27_0906}
 We resume the study of Heegner points introduced in \S\ref{subsection_shimura_curves}, retaining the notation and hypotheses therein. 
In particular: 
 \begin{itemize}
\item The hypothesis \eqref{item_Loc} is assumed.
     \item For a positive integer $n$ and an $n$-admissible set $S$ relative to $f$, we have the Heegner class $ \kappa(S,n)_m\in H^1_{{\rm ord}_S,\emptyset}(K_m,T_{f,n})$.
 \end{itemize}

\subsubsection{Construction of the signed classes} 
In this subsection, we introduce and establish basic properties of plus/minus Heegner points.
It builds on \cite[Section 4]{darmoniovita}, extending the latter to the case $p$ being inert in $K/\QQ$.  

 For a positive integer $m$, let $\omega_m$, $\omega_m^\pm$, and $\widetilde{\omega}_m^\pm$ be the polynomials as in \S\ref{sec:coleman}.  
\begin{lemma}
    \label{eqn_lemma_DILemma4_2}
    Fix a positive integer $n$ and an $n$-admissible set $S$ relative to $f$. For any positive integer $m$, let us denote by $\bullet$ the sign of $(-1)^m$. Then $\omega_m^\bullet \kappa(S,n)_m=0$. 
\end{lemma}

\begin{proof}
In light of the relation \eqref{eqn_Heegner_trace_relation}, the proof is identical to that of \cite[Lemma 4.2]{darmoniovita}.
\end{proof}

\begin{proposition}
        \label{eqn_prop_DIProp4_3}
In the notation of Lemma~\ref{eqn_lemma_DILemma4_2}, suppose that $S' \in \fN_n$ is core-$n$-admissible and is divisible by $S$. There exists a unique element
$$\eta(S,n)_m^\bullet\in H^1_{S',\emptyset}(K_m,T_{f,n})\Big{/}(\omega_m^{\bullet})$$
that does not depend on the choice of $S'$, with the property that 
$$\widetilde{\omega}_m^{\circ}\,\eta(S,n)_m^\bullet=\kappa(S,n)_m\,,$$
where   $\circ$ is the sign of $(-1)^{m+1}$.
\end{proposition}

\begin{proof}
The existence of a core-$n$-admissible $S'$ that is divisible by $S$ follows from Proposition~\ref{prop_2023_04_17_1447}. 

In the next paragraph, we show that $ H^1_{S',\emptyset}(K_m,T_{f,n})$ is a free $\LL_{m,n}$-module.  
Our argument is identical to the proof of \cite[Proposition 3.21]{darmoniovita}.

It follows from \cite[Proposition 3.20]{darmoniovita} that $H^1_{S',0}(K_m,T_{f,n})$ is a free $\LL_{m,n}$-module of rank $\#S'-2$, and from Proposition 3.19 of \textit{op. cit.} that 
\begin{equation*}
   \# H^1_{S',\emptyset}(K_m,T_{f,n})\,=\,\#H^1_{S',0}(K_m,T_{f,n})\,\cdot\, \#H^1(K_{m,p},T_{f,n})\,,
\end{equation*} 
where $H^1(K_{m,p},T_{f,n}):=\bigoplus_{\p\mid p} H^1(K_{m,\p},T_{f,n})$. Note that the proofs of these statements in \cite{darmoniovita} do not rely on the splitting behaviour of $p$ in $K/\QQ$. 
Consequently, we have an exact sequence 
\begin{equation}
    \label{eqn_2022_07_04_1536}
0\lra   H^1_{S',0}(K_m,T_{f,n}) \lra H^1_{S',\emptyset}(K_m,T_{f,n})\lra  H^1(K_{m,p},T_{f,n})\lra 0\,.
\end{equation}
As $H^1(K_{m,p},T_{f,n})$ is a free $\LL_{m,n}$-module, this exact sequence splits and shows that $H^1_{S',\emptyset}(K_m,T_{f,n})$ is also a free $\LL_{m,n}$-module, as required.

The existence and uniqueness of $\eta(S,n)_m^\bullet$ follows from the above fact and Lemma~\ref{eqn_lemma_DILemma4_2}. It does not depend on the choice core-$n$-admissible $S'$: it suffices to check that the natural map
\begin{equation}
\label{eqn_2023_30_05_1615}
H^1_{S',\emptyset}(K_m,T_{f,n})\big{/}(\omega_m^{\bullet}) \lra H^1_{S'',\emptyset}(K_m,T_{f,n})\big{/}(\omega_m^{\bullet})
\end{equation}
is injective for any core-$n$-admissible $S''$ that is divisible by $S'$.

Since $S'$ and $S''$ are both core-$n$-admissible and $S'$ divides $S''$, the sequence
$$0\lra H^1_{S',\emptyset}(K_m,T_{f,n}) \lra H^1_{S'',\emptyset}(K_m,T_{f,n})\xrightarrow{\,\res_{S'',S'}\,} \underbrace{\bigoplus_{\substack{q\mid S''\\ q\nmid S'}}H^1(K_{m,q},T_{f,n})}_{M(S'',S')}\lra 0$$
is exact. Indeed, the cokernel of $\res_{S'',S'}$ is isomorphic to a submodule of ${\rm Sel}_{S',0}(K_m,A_{f,n})^\vee$ by the Poitou--Tate global duality. By a control theorem  (cf. \cite[Lemma 3.5.3]{mr02}), we have 
$${\rm Sel}_{S',0}(K_m,A_{f,n})^\vee_{\Gamma}\simeq {\rm Sel}_{S',0}(K,A_{f,n})^\vee=\{0\}\,,$$ 
where the vanishing follows from the core-$n$-admissibility of $S'$. This fact combined with Nakayama's lemma shows that ${\rm Sel}_{S',0}(K_m,A_{f,n})^\vee=0$, as required. 

Since $S''\in \mathfrak{N}_n$ is $n$-admissible, the $\LL_{m,n}$-module $M(S'',S')$ is free, and therefore flat. This shows that the map \eqref{eqn_2023_30_05_1615} is injective, and the proof concludes.
\end{proof}

\begin{defn}
    \label{def_2023_05_30_1624}
    In the notation of Proposition~\ref{eqn_prop_DIProp4_3}, we define the signed Heegner classes 
    $$\kappa(S,n)_m^\bullet \in H^1_{S',\{\}}(K_m,T_{f,n})\Big{/}(\omega_m^\bullet)\,,\qquad \bullet\in \{+,-\}$$ 
    by  
\begin{align*}
    \kappa(S,n)_m^+&:= (-1)^{\frac{m}{2}}\eta(S,n)_m^+ \qquad \hbox{if $m$ is even}\, ,\\
    \kappa(S,n)_m^-&:= (-1)^{\frac{m+1}{2}}\eta(S,n)_m^- \qquad \hbox{if $m$ is odd}\,.
\end{align*}
\end{defn}

Arguing as in the proof of \cite[Lemma 2.9]{darmoniovita}, it follows that 
$$\kappa(S,n)^\pm:=\{\kappa(S,n)_m^\pm\,:\,(-1)^m=\pm 1 \}_m \in \varprojlim_{(-1)^m=\pm 1} H^1_{S',\emptyset}(K_m,T_{f,n})\Big{/}(\omega_m^\pm)=\widehat{H}^1_{S',\emptyset}(K_{\infty},T_{f,n}),$$
where the inverse limit is with respect to the corestriction maps. 

\subsubsection{}
\label{subsubsec_632_2023_09_19}
For $\bullet\in\{+,-\}$, combining the discussion in \S\ref{subsec_8_1_2022_09_27_0906} with \eqref{eqn_2023_04_17_1607}, we now have 
elements
    \[
    \{\kappa(S,n)^\bullet:S\in\fN_n^\inde\},\quad \{\lambda(S,n)^\bullet\in \Lambda_n:S\in\fN_n^\de\}\,.
    \]
Our goal is to verify that they satisfy the properties listed in Definition~\ref{defn_6_1_2023_04_17} so that Theorem~\ref{thm_bipartite_abstract_main} applies with $\kappa(1)^\bullet$ and~$\lambda(S,n)^\bullet$ in the place of $\hbox{\upshape\texttt K}(1)^\bullet$ and ~$\hbox{\upshape\texttt L}(S,n)^\bullet$, respectively. 

This amounts to verifying: 
\begin{itemize}
    \item[(\mylabel{item_T1}{\textbf{T1}})] For every positive integer $n$ and $S\in \fN_n^\inde$, we have  $\kappa(S,n)^\bullet\in \widehat{H}^{1,\bullet}_{{\rm ord}_S}(K_\infty,T_{f,n})$. 
        \item[(\mylabel{item_T2}{\textbf{T2}})]  For any $S\ell\in \fN_n^\inde$, there is an isomorphism of $\Lambda$-modules
        \[
        \varprojlim_m H^1_\ord (K_{m,\ell},\Tfn)\simeq \Lambda_n
        \]
        sending the image of $\res_\ell(\kappa(S\ell,n)^\bullet)$ in $ \varprojlim_m H^1_\ord (K_{m,\ell},\Tfn)$ to $\lambda(S,n)^\bullet$.
         \item[(\mylabel{item_T3}{\textbf{T3}})] For any $S\ell\in \fN_n^\de$, there is an isomorphism of $\Lambda$-modules
        \[
        \varprojlim_m H^1_\ur(K_{m,\ell},\Tfn)\simeq \Lambda_n
        \]
        sending $\res_\ell(\kappa(S,n)^\bullet)$ to $\lambda(S\ell,n)^\bullet$.
 \item[(\mylabel{item_T4}{\textbf{T4}})] $\kappa(1)^\bullet\neq 0$.
        
\end{itemize}


\subsubsection{Local properties of signed Heegner classes}
\label{sec_Heegner_local_properties_at_p}
The aim of this subsection
is to verify the property \eqref{item_T1}. 

\begin{lemma}
\label{lemma_local_conditions_Heeg_inert_case}
For every positive integer $n$ and $S\in \fN_n^\inde$, we have the following. 
\begin{itemize}
\item[i)] The restriction $\res_p(\kappa(S,n)^\pm)$ belongs to $ \widehat{H}^{1,{\pm}}(K_{\infty,p},T_{f,n})$.
\item[ii)] For any prime $q$ of $K$ that does not divide $pS$, we have
$\res_{q}(\kappa(S,n)^\pm)\in \widehat{H}^1_{\rm ur}(K_{\infty,q},T_{f,n})$.
\item[iii)] For any prime $q$ of $K$ that divides $S$, we have
$\res_{q}(\kappa(S,n)^\pm)\in \widehat{H}^1_{\rm ord}(K_{\infty,q},T_{f,n})$.
\end{itemize}
As a result, the property \eqref{item_T1} holds:
$$\kappa(S,n)^\pm\in  \widehat{H}^{1,{\pm}}_{{\rm ord}_S}(K_{\infty},T_{f,n})\,.$$
\end{lemma}

\begin{proof}
\begin{itemize}
\item[i)] We prove the assertion for $\kappa(S,n)^+$ by an argument which also applies to $\kappa(S,n)^-$. 

Let $S'$ be a core-$n$-admissible integer divisible by $S$ as in the statement of Proposition~\ref{eqn_prop_DIProp4_3}. Note that 
$$\kappa(S,n)^+=\{\kappa(S,n)^+_m\}\in \varprojlim_{m:\, {\rm even}} {H}^1_{S',\emptyset}(K_m,T_{f,n})/(\omega_m^+)=\widehat{H}^1_{S',\emptyset}(K_\infty,T_{f,n})\,,$$
The assertion, therefore, amounts to 
\begin{equation}
    \label{eqn_2022_09_27_1541}
    \res_p\left(\kappa(S,n)^+_m\right)\in {H}^{1,{\pm}}(K_{m,p},T_{f,n})/(\omega_m^+)
\end{equation}
for any positive even integer $m$. To see this, observe that 
$$\widetilde \omega_m^-\res_p\left(\kappa(S,n)^+_m\right)=\res_p\left(\widetilde \omega_m^-\kappa(S,n)^+_m\right)=(-1)^{\frac{m}{2}}\res_p(\kappa(S,n)_m)\in H^1_{\rm f}(K_{m,p},T_{f,n})$$
by \cite[Proposition 6.5]{BBL1} and the construction of Heegner classes as the Kummer images (cf. \cite[\S7]{BertoliniDarmon2005}). Moreover, since $H^1_{\rm f}(K_{m,p},T_{f,n})\subset H^{1,+}(K_{m,p},T_{f,n})$ by definition, we deduce that 
\begin{equation}
    \label{eqn_2022_09_27_1521}
    \widetilde \omega_m^-\res_p\left(\kappa(S,n)^+_m\right)\in H^{1,+}(K_{m,p},T_{f,n})\,.
\end{equation}

Consider the following exact sequence of $\LL_{m,n}$-modules:  
\begin{equation}
    \label{eqn_2022_09_27_1522}
     0\lra H^{1,+}(K_{m,p},T_{f,n})\lra H^{1}(K_{m,p},T_{f,n}) \lra  H^1_{/+}(K_{m,p},T_{f,n}) \lra 0\,,
\end{equation}
where the right-most module is just defined by the exactness.
Applying the functor $(-)\otimes \LL/(\omega_m^+)$ 
to \eqref{eqn_2022_09_27_1522}, we obtain the exact sequence
\begin{equation}
    \label{eqn_2022_09_27_1527}
   H^{1,+}(K_{m,p},T_{f,n})/(\omega_m^+)\lra H^{1}(K_{m,p},T_{f,n})/(\omega_m^+) \lra  H^1_{/+}(K_{m,p},T_{f,n})/(\omega_m^+) \lra 0\,.
\end{equation}
Furthermore, we have the following commutative diagram with exact rows:
\begin{equation}
    \label{eqn_2022_09_27_1529}
    \begin{aligned}
    \xymatrix{
        & H^{1,+}(K_{m,p},T_{f,n})/(\omega_m^+)\ar[r]^{f_1}\ar@{^{(}->}[d]_{\times\, \widetilde \omega_m^-}^{v_1}& H^{1}(K_{m,p},T_{f,n})/(\omega_m^+) \ar[r]^{f_2}\ar@{^{(}->}[d]_{\times\, \widetilde \omega_m^-}^{v_2}&  H^1_{/+}(K_{m,p},T_{f,n})/(\omega_m^+) \ar[r] \ar@{-->}[d]^{v}& 0\\
        0\ar[r]&H^{1,+}(K_{m,p},T_{f,n})\ar[r]_{g_1}& H^{1}(K_{m,p},T_{f,n}) \ar[r]_{g_2} &  H^1_{/+}(K_{m,p},T_{f,n}) \ar[r]& 0\,.
        }
    \end{aligned}
\end{equation}
Note that the vertical maps in the middle and on the left are given by multiplication by $\widetilde \omega_n^-$ and they are injective since the $\LL_{m,n}$-modules $H^{1,+}(K_{m,p},T_{f,n})$ and $H^{1}(K_{m,p},T_{f,n})$ are both free by \cite[Lemma~4.12]{BBL1}. The dotted vertical arrow $v$ is induced from the exactness of the first row and the commutativity of the square on the left. 

We would like to prove, relying on \eqref{eqn_2022_09_27_1521}, the containment  \eqref{eqn_2022_09_27_1541}, which is equivalent to the assertion that 
$$\res_p(\kappa(S,n)^+_m)\in {\rm im}(f_1)=\ker(f_2)\,.$$ 
Chasing the diagram \eqref{eqn_2022_09_27_1529}, this follows once we check that the vertical map $v$ in this diagram is injective, which in turn is equivalent to, by the snake lemma, that the induced map
$$H^{1,+}(K_{m,p},T_{f,n})/(\widetilde \omega_m^-)={\rm coker}(v_1)\lra {\rm coker}(v_2)=H^{1}(K_{m,p},T_{f,n})/(\widetilde \omega_m^-)$$
is injective. The latter follows from the following commutative diagram, where the vertical maps are injective since the $\LL_{m,n}$-modules $H^{1,+}(K_{m,p},T_{f,n})$ and $H^{1}(K_{m,p},T_{f,n})$ are both free:
\begin{equation}
    \label{eqn_2022_09_27_1558}
    \begin{aligned}
    \xymatrix{
     H^{1,+}(K_{m,p},T_{f,n})/(\widetilde \omega_m^-) \ar[r]\ar@{^{(}->}[d]_{\times\, \omega_m^+}& H^{1}(K_{m,p},T_{f,n})/(\widetilde \omega_m^-) \ar@{^{(}->}[d]_{\times\, \omega_m^+}\\
        H^{1,+}(K_{m,p},T_{f,n})\ar@{^{(}->}[r]&  H^{1}(K_{m,p},T_{f,n})\,.
        }
    \end{aligned}
\end{equation}
\item[ii)] For primes $q|N_0$, the assertion just follows from \eqref{eqn_construct_Heeg_via_level_raising}.

So it suffices consider the case  $q\mid S'$, where  
$S'$ is a core-$n$-admissible integer divisible by $S$ as in the statement of Proposition~\ref{eqn_prop_DIProp4_3}. 
For such $q$, the assertion follows by the argument for (i), relying on the freeness of the $q$-local cohomology as in Lemma~\ref{lemma_howard_06_221}. Alternatively, we can set $S'$ in the statement of Proposition~\ref{eqn_prop_DIProp4_3} to be an integer that is not divisible by $q$. The aforementioned proposition then implies the desired inclusion for all $q\nmid pS$.
\item[iii)] One proves this part also following the argument in (ii) and utilizing the freeness assertion in  Lemma~\ref{lemma_howard_06_221}.
\end{itemize}
\end{proof}

When $S=1$ is the empty product, we have
$$\kappa(1,n)^\bullet \in \widehat{H}^1_{\{\}}(K_\infty,\Tfn)$$
for all positive integers $n$, and 
$$\kappa(1)^\bullet:=\{\kappa(1,n)^\bullet\}_n \in \varprojlim_n\widehat{H}^{1}_{\{\}}(K_\infty,\Tfn)=H^1(K,\TT_{f}^\ac)\,.$$

\subsubsection{Reciprocity laws} The aim of this subsection is to verify the properties \eqref{item_T2} and \eqref{item_T3}. 

Fix positive integers $m$ and $n$, as well as $S\in \fN_n$, a prime $\ell \in \fN_n$ that is coprime to $S$. It follows from Lemma~\ref{lemma_howard_06_221} that we have isomorphisms
\begin{equation}
    \label{eqn_2023_05_31_1217}
    \partial_\ell\,:\,H^1_{\rm ord}(K_{m,\ell},T_{f,n})\xrightarrow{\,\sim\,} \LL_{m,n}
\end{equation}
\begin{equation}
    \label{eqn_2023_05_31_1220}
    v_\ell\,:\,H^1_{\rm ur}(K_{m,\ell},T_{f,n})\xrightarrow{\,\sim\,} \LL_{m,n}
\end{equation}
which are determined by the choice of a topological generator of the tame inertia group $I_{K_\ell}^{\rm t}$ and an $\cO_L/(\varpi^n)$-module basis of $T_{f,n}^{{\rm Fr}_{(\ell)}=\ell^2}$, where ${\rm Fr}_{(\ell)}\in G_{K_\ell}/I_{K_\ell}$ is the Frobenius (cf. \cite[Corollary~7.6]{BBL1}). We will also write $\partial_\ell$ (resp. $v_\ell$) in place of $\partial_\ell\circ \res_\ell$ (resp. $v_\ell\circ \res_\ell$).

First, assume that $S\ell \in \fN_n^\inde$. As utilized in \cite[Proof of Proposition 4.4]{darmoniovita}, the proof of \cite[Theorem 4.1]{BertoliniDarmon2005} in \S8 of {\it loc. cit.}  can be adapted\footnote{See also \cite[Construction 5.1.2]{sweeting}.} to the non-ordinary setting  to give
$$\partial_\ell(\kappa(S\ell, n)_m)\,\dot{=}\, \lambda(S,n)_m\,.$$
Here $''\dot{=}''$ denotes equality up to multiplication by an element $\LL_{m,n}^\times$ and we note that there is no assumption in \cite{BertoliniDarmon2005} on the splitting behaviour of the prime $p$ in $K/\QQ$. It follows from the definitions that we have 
$$\partial_\ell(\kappa(S\ell, n)_m^\pm)\,=\, \lambda(S,n)_m^\pm $$
$\LL_{n}/(\omega_m^\pm)$, 
up to multiplication by units in the ring $\LL_{n}/(\omega_m^\pm)$. The proof of \eqref{item_T2} follows by passing to limit in $m$ with appropriate parity, which is determined by the choice of the sign $\pm$.

Suppose now that $S \in \fN_n^\inde$. The verification of \eqref{item_T3} proceeds in an identical way, starting with the equality
$$v_\ell(\kappa(S, n)_m)\,\dot{=}\, \lambda(S\ell,n)_m \quad \mod \LL_{m,n}^\times$$
proved in \cite[Theorem 4.2]{BertoliniDarmon2005}. (See also \cite[Construction~5.1.2]{sweeting} for a more general form.) Note that these results apply regardless of the reduction type of $E$ at the prime $p$ or the splitting behaviour of $p$ in $K/\QQ$.

\subsubsection{} The desired non-vanishing 
\eqref{item_T4} of signed Heegner classes is a consequence of Mazur's conjecture established in \cite{cornut}. Indeed, proceeding as in the proof of \cite[Thm.~2.5]{pollack-weston11}, it suffices to show that the Heegner points $\{\kappa(1)_{m}\}_{m}$ are non-torsion for infinitely many $m$, which is the main result of \cite{cornut}.

\begin{remark}
For a refinement of Mazur's conjecture for primes $p$ split in $K$, the reader may refer to  \cite{burungale2020,burungale2017}. The non-split analogue will be considered in a future work.  
\end{remark}

\subsection{Proof of Theorem~\ref{thm:main}} 
In view of the construction of signed Heegner point bipartite Euler systems  in \S\ref{subsec_8_1_2022_09_27_0906}, Theorem~\ref{thm_bipartite_abstract_main} yields parts (i)--(ii) of Theorem~\ref{thm:main}. More precisely, we have the following.

\begin{theorem}
\label{thm_bipartite_Heegner_main}
Let $E/\QQ$ be a semistable elliptic curve of conductor $N_{0}$ and $f$ the associated elliptic newform, $p\geq 5$ a prime of good supersingular reduction. Let $K$ be an imaginary with discriminant coprime to $N_0$ satisfying \eqref{item_GHH}. Suppose also that $p$ is inert in $K$, and satisfies \eqref{item_Loc}.
Let $\bullet\in \{+,-\}$. 
\begin{itemize}
    \item[i)]We have $$\rank_\Lambda\, H^{1,\bullet}(K,\TT_f^\ac)=\rank_\Lambda\,\Sel_\bullet(K_\infty,A_f)^\vee=1.$$
    \item[ii)] 
    We have
    $${\rm char}_{\LL}\left(H^{1,\bullet}(K,\TT_{f}^\ac)/\LL\cdot\kappa(1)^\bullet \right)^2 \subset {\rm char}_{\LL}(\Sel_\bullet(K_\infty,A_f)^\vee_{\rm tor})\,.$$
    \item[iii)] The inclusion in part (ii) is an equality upon localisation at a height one prime $\fP$ of $\LL$ if the following holds: There exists a positive integer $k_0=k_0(\mathfrak{P})$ with the property that for an integer $n\geq k_0$, there exists an integer $S\in \fN_n^{\de}$ such that $\lambda(S,n)^\bullet \in \LL_n$ has non-zero image in $\LL/(\mathfrak{P},p^n)$. In particular, if the latter hypothesis is satisfied by all height one prime ideals $\fP$ of $\LL$, then the inclusion in part (ii) is an equality.     
\end{itemize}
\end{theorem}

For part (iii) of Theorem~\ref{thm:main}, it is a consequence of the following proposition verifying the equality criterion of Theorem~\ref{thm_bipartite_Heegner_main}(iii) for $\bullet=-$.

\begin{proposition}\label{prop:equality}
    Suppose that $E$ is semistable, and that the hypothesis \eqref{item_Loc} holds. 
    Then the criterion in part (iii) of Theorem~\ref{thm_bipartite_Heegner_main} is satisfied for $\bullet=-$ by any height one prime ideal of $\Lambda$.
\end{proposition}
\begin{proof}
The following argument proceeds along the line of \cite[proof of Proposition~3.7]{BCK} (see also~\cite[proof of Theorem~A.1]{CHKLL}). 
   
Put $r=\dim_{\Fp}\Sel(K,A_{f,1})$. By the parity conjecture, as established in \cite{BDKim_Parity,DD_parity},  note that $r$ is an odd integer. It then follows from Proposition~\ref{prop_sweeting_lifting}  that there exist $r'$ admissible primes\footnote{An analysis of the proof of Proposition~\ref{prop_sweeting_lifting} shows that $r'$ can be taken to be $r$.} $q_1,\ldots,q_{r'}$ and an elliptic newform $g=g_S\in S_2(\Gamma_0(N_0S))$ for $S=q_1\cdots q_{r'}$, such that $T_{g,1}$ and $T_{f,1}$  are isomorphic $G_\QQ$-representations, and 
\[\Sel(K,A_{g})=0.\]
Therefore, the $p$-part of the Birch and Swinnerton-Dyer formula for $A_g$ over $K$ as established\footnote{These results require the level to be square-free.} in \cite[Theorem 1.5]{bstw} and \cite[Theorem~C]{CCSS} implies that $$\frac{L(g/K,1)}{\Omega_g^{\mathrm{can}} \eta_{g,N^+,N^-S}}$$ is a $p$-adic unit. Indeed, the BSD formula implies that $$\frac{L(g/K,1)}{\Omega_{g}^{\rm can}\prod_{\ell \mid N_{0} S}c_{\ell}}$$ is a $p$-adic unit where $c_\ell$ denotes the Tamagawa number of $g$ over $K$. On the other hand, the $p$-adic valuation of $\prod_{\ell \mid N^{-}S}c_{\ell}$ is the same as  that of $\eta_{g,N^{+},N^{-}S}$ by \cite[Theorem~6.4]{zhang14}, and $\prod_{\ell\mid N^{+}}c_{\ell}$ is a $p$-adic unit by the first case of \eqref{item_Loc}.  
Hence, Lemma~\ref{lem:interpolation} concludes the proof of (iii). 
\end{proof}

\section{The split case}\label{appA}

In this section, we generalize the results of the previous sections to the non-ordinary case when $p$ is split in $K/\QQ$ and $a_p(f)$ is not necessarily zero. The main result is Theorem \ref{thm:appendix} below.  
While the strategy is similar, some of the aspects of its elaboration differ (cf.~\S\ref{pmth}). 
\subsection{Setting}
We retain the preceding notation. 

In particular, $K$ is an imaginary quadratic field. Let $p\geq 5$ be a prime split in $K$. Let $K_\infty$ be the anticyclotomic $\ZZ_p$-extension of $K$ and $\Gamma=\Gal(K_\infty/K)$. Fix embeddings $\iota_\infty: \overline{\QQ}\hookrightarrow \CC$ and $\iota_p:\overline{\QQ}\hookrightarrow \overline{\QQ}_p$. 

Throughout, $f$ is a weight two newform of level $N_0=N^+N^-$, where $N^+$ (resp. $N^-$) is only divisible by primes which are split (resp. inert) in $K$. Suppose that the hypothesis \eqref{item_GHH} holds.
Let $F$ be the Hecke field generated by the Fourier coefficients of $f$. Unlike in previous sections, we do not assume that $F=\QQ$. Fix a prime $v$ of $F$ lying above $p$ arising from $\iota_p$ and let $L$ be the completion of $F$ at $v$. The ring of integers of $L$ is denoted by $\cO_L$. Fix   a uniformizer $\varpi$ of $L$. Assume that 
\[
\ord_\varpi(a_p(f))>0.
\]
In particular, this includes the case where $a_p(f)=0$.

Let $\rho_f$ be the associated $p$-adic Galois-representation. We assume the following hypothesis:
\begin{itemize}
       \item[(\mylabel{item_BI'}{\textbf{Im}})]  
        The residual Galois representation $\bar{\rho}_f$ 
        is surjective.
 \item[(\mylabel{item_Loc'}{\textbf{ram}})]
        If $\ell\mid N_0$, then $\bar{\rho}_E$ is ramified at $\ell$ in either of the following cases:
        \begin{itemize}
        \item[\tiny{$\circ$}] $\ell\mid N^{+}$,
        \item[\tiny{$\circ$}] $\ell \mid N^-$ and $\ell^2\equiv 1\mod{p}$.
        \end{itemize}

        \end{itemize}
In addition, as in \cite{BBL1}, assume:
\begin{itemize}
    \item [(\mylabel{item_iso}{\textbf{iso}})] If $a_{p}(f)\neq 0$, then the newform $f$ is $p$-isolated.
\end{itemize}

Let $\cA_f$ be a $\GL_2$-type abelian variety associated with $f$. Let $T_f$ be the $v$-adic Tate module of $\cA_f$. Put $V_f=T_f\otimes_{\cO_L}L$ and $A_f=V_f/T_f$. Given an integer $n\ge0$, write $\Tfn=T_f/\varpi^n T_f$ and $\Afn=A_f[\varpi^n]$. 

Let $\Lambda$ denote the Iwasawa algebra $\cO_L[[\Gamma]]$, where $\Gamma=\Gal(K_\infty/K)$. Given integers $m,n\ge0$, let $\Lambda_n=\Lambda/(\varpi^n)$ and $\Lambda_{m,n}=\cO_L/(\varpi^n)[G_m]$. 
\subsection{Coleman maps and local conditions}
Let $w$ be a place of $K_\infty$ lying above $p$. Let $\fp$ be the place of $K$ lying below $w$. Let $\Gamma_w=\Gal(K_{\infty,w}/K_\fp)\cong\Zp$ be the decomposition group of $w$ in $\Gamma$. To lighten notation, we shall write $k_\infty=K_{\infty,w}$. 

For an integer $ m\ge0$, let $k_m$  denote the intermediate extension of $k_\infty/K_v$ that is of degree $p^m$. Let $\cG_m=\Gal(k_m/k_0)$ and define
$$\mathfrak{R}=\cO_L[[\Gamma_w]]=\varprojlim \cO_L[\cG_m].$$ 
Given an integer $n\ge 1$,  write $\mathfrak{R}_n=\mathfrak{R}/p^n \mathfrak{R}=(\cO_{L}/p^n\cO_{L})[[\Gamma_w]]$. If $m\ge0$ is another integer, let $\mathfrak{R}_{m,n}$ denote the group ring $(\cO_{L}/p^n\cO_{L})[\cG_m]$.

As in \cite[\S4.2]{BBL1}, one may construct a primitive $Q$-system  $(d_m)_{m\ge0}$, where $d_m\in H^1_\mathrm{f}(k_m,\Tfn)$ satisfying the norm relation
\[
\cor_{k_{m+1}/k_m}(d_{m+1})-a_p(f)d_m+\res_{k_{m-1}/k_m}(d_{m-1})=0\ \forall m\ge1.
\]
It leads to surjective Coleman maps
\[
\col^\bullet_{\Tfn}:\HIw(k_\infty,\Tfn)\rightarrow \fR_n,\ \binsf.
\]
More concretely, let $C_{f,m}=\begin{pmatrix}
    a_p(f)&1\\-\Phi_m&0
\end{pmatrix}$ and write $H_{f,m}:\fR_{m,n}^2\rightarrow \fR_{m,n}^2$ for the $\fR_n$-morphism given by 
$$\begin{pmatrix}
    x\\y
\end{pmatrix}\longmapsto C_{f,m}\cdots C_{f,1}\begin{pmatrix}
    x\\y
\end{pmatrix}.$$ 
Let us denote by 
\[
\langle-,-\rangle_{m,n}\,:\, H^1(k_m,\Tfn)\times H^1(k_m,\Tfn)\lra  \cO_L/(\varpi^n)
\]
the usual local pairing. We have
\[
H_{f,m}\begin{pmatrix}
    \col^\sharp\\\col^\flat
\end{pmatrix}\equiv \begin{pmatrix}
    \col_{\bd,m}\\ -\xi_{m-1}\col_{\bd,m-1}
\end{pmatrix}\mod\omega_m,
\]
where $\col_{\bd,m}:H^1(k_m,\Tfn)\rightarrow \fR_{m,n}$ is given by
\[
z\longmapsto \sum_{\sigma\in G_m}\langle z^{\sigma},d_m\rangle_{m,n}\cdot \sigma^{-1}
\]
and $\xi_{m-1}:\fR_{m-1,n}\rightarrow \fR_{m,n}$ is the norm map as defined in \S\ref{S:norm-relations}.

Recall furthermore that these maps are compatible as $n$ varies (see \cite[Corollary 4.5, Lemma 4.6 and Proposition 4.7]{BBL1}). We may define local conditions at primes above $p$ using the kernels of these maps, and hence Selmer groups, as in \S\ref{sec:Sel} and \cite[\S7]{BBL1}. 

\subsection{Signed theta elements}\label{pmth}
The discussion in \S\ref{S:Shimura}, up until the end of \S\ref{subsection_shimura_curves} 
applies to the current setting. Let $f_S$ be defined as in \eqref{eq:f_S}. 

Since $a_p(f_S)$ is no longer assumed to be zero, we have the following more general norm relation:
\begin{equation}
        \label{eq:norm-7.1}\pi_{m+1,m}\lambda(S,n)_{m+1}=a_p(f_S)\lambda(S,n)_{m}-\xi_{m-1}\lambda(S,n)_{m-1}\,,\qquad \forall\,m\in \ZZ^+\,,
\end{equation}
 replacing \eqref{eqn_2023_04_17}, where $\xi_{m-1}$ is the norm map defined in \S\ref{S:norm-relations}. In turn, we replace Lemmas~\ref{lemma_DI_Prop_2_8} and \ref{lemma_DI_Lemma_2_9} with the following.

\begin{lemma}\label{lem:sharpflatpadicL}
    Let  $n$ be a positive integer. There exist  $\lambda(S,n)^\sharp$ and $\lambda(S,n)^\flat$ in $\cO_L/(p^n)[[\Gamma]]$ such that
    \[
    \begin{pmatrix}
        \lambda(S,n)_m\\ -\xi_{m-1}  \lambda(S,n)_{m-1}
    \end{pmatrix}
    \equiv H_{f,m}
    \begin{pmatrix}
        \lambda(S,n)^\sharp\\ \lambda(S,n)^\flat
    \end{pmatrix}\mod (\omega_m)
    \]
    for all $m\ge0$.
\end{lemma}
\begin{proof}
    This follows from \cite[Theorem~3.4]{BBL1}.
\end{proof}
For $\bullet\in\{\sharp,\flat\}$, we write $\lambda(S,n)_m^\bullet$ to be the natural image of $\lambda(S,n)^\bullet$ in $\cO_L/(p^n)[G_m]$.

Another diverging point of the split case from the inert case treated in the previous sections is that Lemma~\ref{lem:interpolation} does not carry over. Instead, the following holds.

\begin{lemma}\label{lem:app-interpolation}
 Let  $g$ be a characteristic zero form lifting $f_S$. If $\displaystyle\dfrac{L(g/K,1)}{\Omega_g^{\mathrm{can}} \eta_{g,N^+,N^-S}}$ is a $p$-adic unit, then $\lambda(S,n)^\bullet(\one)$ is a $p$-adic unit for both $\bullet\in\{\sharp,\flat\}$. 
    \end{lemma}
\begin{proof}

 Let $\alpha$ and $\beta$ be the two roots of the Hecke polynomial of $g$ at $p$. Denote by $\lambda(g)_{\alpha,m}$ and $\lambda(g)_{\beta,m}$ the corresponding theta elements attached to the $p$-stabilizations of $g$ with respect to these two roots, given as in \cite[Definition~4.1]{chidahseihanticycloLvaluescrelle}. 
 
 As in the proof of Lemma~\ref{lem:interpolation}, we have
\[\lambda(g)_{\clubsuit,1}(\one)=\lambda(g)_{\clubsuit,0}(\one)\]
for $\clubsuit\in\{\alpha,\beta\}$.
Since $p$ is assumed to be split in $K$, \cite[Theorem~4.6]{chidahseihanticycloLvaluescrelle} says that
\begin{equation}
    \lambda(g)_{\clubsuit,0}(\one)^2=\lambda(g)_{\clubsuit,0}(\one)^2=u(1-\clubsuit^{-1})^2\cdot \frac{L(g/K,1)}{\Omega_g^{\mathrm{can}} \eta_{g,N^+,N^-S}}\label{eq:CH-interpolation}
\end{equation}
for some $p$-adic unit $u$.

Recall the relations
\begin{align*}
 \alpha^2\lambda(g)_{\alpha,1}-\beta^2\lambda(g)_{\beta,1}&=(\alpha-\beta)\lambda(g)_1,\\
\alpha\lambda(g)_{\alpha,0}-\beta\lambda(g)_{\beta,0}&=(\alpha-\beta)\lambda(g)_0.   
\end{align*}
Put 
$$\fu:=u\cdot \dfrac{L(g/K,1)}{\Omega_g^{\mathrm{can}} \eta_{g,N^+,N^-S}},\ 
\fv_\clubsuit:=\lambda(g)_{\clubsuit,1}(\one).
$$ We then have 
\begin{align*}
   (\alpha-\beta)^2 \lambda(g)_1(\one)^2&=(\alpha^4+\beta^4-2\alpha^3-2\beta^3+\alpha^2+\beta^2)\fu-2\alpha^2\beta^2\fv_\alpha\fv_\beta\\
      (\alpha-\beta)^2 \lambda(g)_0(\one)^2&=(\alpha^2+\beta^2-2\alpha-2\beta+2)\fu-2\alpha\beta\fv_\alpha\fv_\beta.
\end{align*}
Combining the two equations to cancel the term $\fv_\alpha\fv_\beta$ gives
\begin{equation}
\lambda(g)_1(\one)^2-\alpha\beta\lambda(g)_0(\one)^2=(1-2(\alpha+\beta)+\alpha^2+\alpha\beta+\beta^2)\fu. \label{eq:combined}   
\end{equation}

Suppose that $\fu$ is a $p$-adic unit. Since $p$ is split, \cite[\S2.4]{BD-mumford} tells us that $x_0$ is equal to $T_p x_0^0-2x_0^0$ up to a $p$-adic unit. Thus, $g(x_0)$, and hence $\lambda(g)_0(\one)$ are $p$-adic units. \setcounter{footnote}{0}
Here we utilize\footnote{In our setting, the new vector is a test vector, and so the toric period of $g$ explicitly relates to the normalized $L$-value. Note that the test vectors in \cite{cst} and \cite{chidahseihanticycloLvaluescrelle} are identical.} the explicit Waldspurger formula \cite[Theorem~1.2]{cst}, which expresses $$\dfrac{L(g/K,1)}{\Omega_g^{\mathrm{can}} \eta_{g,N^+,N^-S}}$$ as a $p$-unit multiple of the square of $g(x_{0})$: this follows in the same way as the deduction of \cite[Corollary 6.2]{zhang14} from the explicit Waldspurger formula \cite[Theorem 6.1]{zhang14} therein. Since $\alpha$ and $\beta$ have positive $p$-adic valuations, it follows from \eqref{eq:combined} that $\lambda(g)_1(\one)$ is a $p$-adic unit.

The elements $\{\lambda(g)_m\}_{m\ge0}$ satisfy a norm relation analogous to \eqref{eq:norm-7.1}. Thus, if we take $m=1$ in the analogous decomposition to Lemma~\ref{lem:sharpflatpadicL}, we have
\[
\begin{pmatrix}
    a_p(g)&1\\
    -{\Phi_1}&0
\end{pmatrix}\begin{pmatrix}
    \lambda(g)_1^\sharp\\ \lambda(g)_1^\flat
\end{pmatrix}=\begin{pmatrix}
    \lambda(g)_1\\ -\xi_1\lambda(g)_0
\end{pmatrix},
\]
which implies
\begin{align*}
\lambda(g)_1^\sharp(\one)&= \lambda(g)_0(\one),\\  
\lambda(g)_1^\flat(\one)&=\lambda(g)_1(\one)-a_p(g)\lambda(g)_0(\one).
\end{align*}
Thus, if $\lambda(g)_0(\one)$ and $\lambda(g)_1(\one)$ are $p$-adic units, so are $\lambda(g)_1^\sharp(\one)$ and $\lambda(g)_1^\flat(\one)$. This concludes the proof.
\end{proof}

\subsection{Main result}

The general machinery of signed bipartite Euler systems introduced in \S\ref{S:generaltheory} and the proof of Theorem~\ref{thm_bipartite_abstract_main}  in \S\ref{subsec_proof_thm_bipartite_abstract_main} 
carries over 
to the current setting. 
Besides the discussion in \S\ref{pmth}, we simply need to replace the symbols $+$ and $-$  by $\sharp$ and $\flat$, respectively. The fact that we work with a more general coefficient ring does not affect the proofs therein.

As $a_p(f)$ is no longer assumed to vanish, the construction of signed bipartite Euler systems carried out in \S\ref{subsec_8_1_2022_09_27_0906} needs to be modified. More specifically, we replace Proposition~\ref{eqn_prop_DIProp4_3} with the following.

\begin{proposition}
    Let $m,n\ge1$ be integers, $S$ an $n$-admissible prime relative to $f$, and $S'\in\fN_n$ which is core-$n$-admissible and divisible by $S$. There exists a unique element
    $$\begin{pmatrix}
        \kappa(S,n)_m^\sharp\\ \kappa(S,n)_m^\flat
    \end{pmatrix}\in \widehat{H}^1_{S',\emptyset}(K_{m},T_{f,n})^{\oplus 2}/\ker(H_{f,m})\cdot \widehat{H}^1_{S',\emptyset}(K_{m},T_{f,n})^{\oplus 2}$$ 
    such that
      \[
    \begin{pmatrix}
        \kappa(S,n)_m\\ -\res_{K_m/K_{m-1}}  \kappa(S,n)_{m-1}
    \end{pmatrix}
    = H_{f,m}
    \begin{pmatrix}
        \kappa(S,n)_m^\sharp\\ \kappa(S,n)_m^\flat
    \end{pmatrix}.
    \]
    Furthermore, these classes are compatible under the corestriction maps as $m$ varies, resulting in classes $\kappa(S,n)^\sharp,\kappa(S,n)^\flat\in\widehat{H}^1_{S',\emptyset}(K_{\infty},T_{f,n})$.
\end{proposition}
\begin{proof}
    This follows from the same proof as \cite[Theorem~8.1 and Lemma~8.3]{BBL1}.
\end{proof}
For $\bullet\in\{\sharp,\flat\}$, let $\kappa(S,n)^\bullet_m$ also denote the image of $\kappa(S,n)^\bullet$ under the natural map $$\widehat{H}^1_{S',\emptyset}(K_{\infty},T_{f,n})\lra\widehat{H}^1_{S',\emptyset}(K_{m},T_{f,n}).$$
We are equipped with two families of elements
\[
    \left\{\kappa(S,n)^\bullet:S\in\fN_n^\inde\right\},\quad \left\{\lambda(S,n)^\bullet\in \Lambda_n:S\in\fN_n^\de\right\}\,
    \]
    as before. In order to apply the bipartite Euler system machinery, it suffices to verify the properties \eqref{item_T1}--\eqref{item_T4}. When $a_p(f)\ne0$, the verification of \eqref{item_T1} in \S~\ref{sec_Heegner_local_properties_at_p} does not carry over directly. Instead, we employ the techniques of \cite[\S9.2]{BBL1}.

    Putting all the ingredients together, we can then prove the analogues of Theorem~\ref{thm_bipartite_Heegner_main} and Proposition~\ref{prop:equality} by the same method, leading to the following. 
    
\begin{theorem}\label{thm:appendix}
Assume that $p\geq5$ is split in $K$. Assume also that the hypotheses \eqref{item_GHH}, \eqref{item_BI'}, \eqref{item_Loc'} and \eqref{item_iso} hold, and that the level $N_0$ is square-free. Let $\bullet\in \{\sharp,\flat\}$. 
\begin{itemize}
    \item[i)] We have $\rank_\Lambda\,\Sel_\bullet(K_\infty,A_f)^\vee=1$\,. 
    \item[ii)] 
    We have 
    $${\rm char}_{\LL}\left(H^{1,\bullet}(K,\TT_{f}^\ac)/\LL\cdot\kappa(1)^\bullet \right)^2 = {\rm char}_{\LL}(\Sel_\bullet(K_\infty,A_f)^\vee_{\rm tor})\,.$$
   \end{itemize}
\end{theorem}
 This can be considered as a natural generalization of \cite[Theorem~A.1]{CHKLL} to the case where $a_p(f)\ne0$. We have the following application.
\begin{corollary}\label{cor:p-conv}
  Let $f$ be a weight two newform of level $\Gamma_0(N_0)$ satisfying \eqref{item_iso}, $p\geq 5$ a non-ordinary prime and $\varpi$ a uniformiser of the completion of the endomorphism ring of the associated abelian variety $\cA_f$ at a prime above $p$. Let $F$ denote the Hecke field of $f$. 
  Suppose that $N_0$ is square-free, $p$ is coprime to the associated Tamagawa numbers
  and the hypothesis \eqref{item_BI'} holds.  Then, 
  $$
{\rm corank}_{\cO_L}\Sel_{\varpi^\infty}(\cA_f/\QQ)=1\ \implies \ \ord_{s=1}L(\cA_f,s)=[F:\QQ].
$$
\end{corollary}
\begin{proof} We proceed as in the proof Corollary~\ref{p-cv}.
 \end{proof}
 
\begin{remark}\label{rk:lastwords} The signed anticyclotomic Iwasawa theory is not yet developed for primes $p$ inert in $K$ when $a_p(f)\neq 0$. 
    A key missing ingredient 
    is the existence of local points ~\cite{BKO1}. In the near future, we hope to construct  local points in $\cA_f(k_m)$ satisfying the desired  trace relation \cite[\S4.2]{BBL1}. 
\end{remark}

\bibliographystyle{alpha}
\bibliography{references}

@article {KimKimSun,
    AUTHOR = {Kim, Chan-Ho and Kim, Myoungil and Sun, Hae-Sang},
     TITLE = {On the indivisibility of derived {K}ato's {E}uler systems and
              the main conjecture for modular forms},
   JOURNAL = {Selecta Math. (N.S.)},
  FJOURNAL = {Selecta Mathematica. New Series},
    VOLUME = {26},
      YEAR = {2020},
    NUMBER = {2},
     PAGES = {Paper No. 31, 47},
      ISSN = {1022-1824,1420-9020},
   MRCLASS = {11R23 (11F67 11Y40)},
  MRNUMBER = {4090585},
MRREVIEWER = {Aprameyo\ Pal},
       DOI = {10.1007/s00029-020-00554-w},
       URL = {https://doi.org/10.1007/s00029-020-00554-w},
}

@article {sakamoto_Gorenstein,
    AUTHOR = {Sakamoto, Ryotaro},
     TITLE = {Stark systems over {G}orenstein local rings},
   JOURNAL = {Algebra Number Theory},
  FJOURNAL = {Algebra \& Number Theory},
    VOLUME = {12},
      YEAR = {2018},
    NUMBER = {10},
     PAGES = {2295--2326},
      ISSN = {1937-0652,1944-7833},
   MRCLASS = {11R23 (11F80 11S25)},
  MRNUMBER = {3911132},
MRREVIEWER = {Meng\ Fai\ Lim},
       DOI = {10.2140/ant.2018.12.2295},
       URL = {https://doi.org/10.2140/ant.2018.12.2295},
}

@incollection {grossparson,
    AUTHOR = {Gross, Benedict H. and Parson, James A.},
     TITLE = {On the local divisibility of {H}eegner points},
 BOOKTITLE = {Number theory, analysis and geometry},
     PAGES = {215--241},
 PUBLISHER = {Springer, New York},
      YEAR = {2012},
   MRCLASS = {11G05 (11G07 11G10)},
  MRNUMBER = {2867919},
MRREVIEWER = {\'{A}lvaro Lozano-Robledo},
       DOI = {10.1007/978-1-4614-1260-1\_11},
       URL = {https://doi.org/10.1007/978-1-4614-1260-1_11},
}

@incollection {edi97,
    AUTHOR = {Edixhoven, Bas},
     TITLE = {Serre's conjecture},
 BOOKTITLE = {Modular forms and {F}ermat's last theorem ({B}oston, {MA},
              1995)},
     PAGES = {209--242},
 PUBLISHER = {Springer, New York},
      YEAR = {1997},
      ISBN = {0-387-94609-8; 0-387-98998-6},
   MRCLASS = {11F80 (11G05)},
  MRNUMBER = {1638480},
}

@article {bf,
    AUTHOR = {Friedberg, Solomon and Hoffstein, Jeffrey},
     TITLE = {Nonvanishing theorems for automorphic {$L$}-functions on
              {${\rm GL}(2)$}},
   JOURNAL = {Ann. of Math. (2)},
  FJOURNAL = {Annals of Mathematics. Second Series},
    VOLUME = {142},
      YEAR = {1995},
    NUMBER = {2},
     PAGES = {385--423},
      ISSN = {0003-486X,1939-8980},
   MRCLASS = {11F70 (11F67 22E55)},
  MRNUMBER = {1343325},
MRREVIEWER = {David\ Ginzburg},
       DOI = {10.2307/2118638},
       URL = {https://doi.org/10.2307/2118638},
}

@article {clw,
    AUTHOR = {Castella, Francesc and Liu, Zheng and Wan, Xin},
     TITLE = {Iwasawa-{G}reenberg main conjecture for nonordinary modular
              forms and {E}isenstein congruences on {GU}(3,1)},
   JOURNAL = {Forum Math. Sigma},
  FJOURNAL = {Forum of Mathematics. Sigma},
    VOLUME = {10},
      YEAR = {2022},
     PAGES = {Paper No. e110, 90},
      ISSN = {2050-5094},
   MRCLASS = {11R23 (11F33 11F67)},
  MRNUMBER = {4522696},
       DOI = {10.1017/fms.2022.95},
       URL = {https://doi.org/10.1017/fms.2022.95},
}

@article {BKO1,
    AUTHOR = {Burungale, Ashay and Kobayashi, Shinichi and Ota, Kazuto},
     TITLE = {Rubin's conjecture on local units in the anticyclotomic tower
              at inert primes},
   JOURNAL = {Ann. of Math. (2)},
  FJOURNAL = {Annals of Mathematics. Second Series},
    VOLUME = {194},
      YEAR = {2021},
    NUMBER = {3},
     PAGES = {943--966},
      }

@article{BKO2,
title={$p$-ADIC {$L$}-FUNCTIONS AND RATIONAL POINTS ON {CM} ELLIPTIC CURVES AT INERT PRIMES},
Volume={23},
number={3},
journal={J. Inst. Math. Jussieu},
author={Burungale, Ashay A. and Kobayashi, Shinichi and Ota, Kazuto}, 
year={2024}, 
pages={1417-1460},
}

@misc{BKO3,
title = {The $p$-adic valuation of local resolvents, generalized Gauss sums and anticyclotomic Hecke L-values of imaginary quadratic fields at inert primes}, 
journal={Amer. J. Math.},
author = {Burungale, Ashay A. and Kobayashi, Shinichi and Ota, Kazuto},
     year = {2026},
     }

@misc{BHKO,
title = {Hecke $L$-values, definite Shimura sets and mod $\ell$ non-vanishing}, 
author = {Burungale, Ashay A. and Wei He and Kobayashi, Shinichi and Ota, Kazuto},
year = {2024},
  note = {Preprint, arXiv:2408.13932},
     }

@misc{bstw,
title = {Zeta elements for elliptic curves and applications},
    author = {Burungale, Ashay A. and Skinner, Christopher and Tian, Ye and Wan, Xin},
     year = {2024},
  note = {Preprint, arXiv:2409.01350},
     }

@misc{BLV,
    title={The anticyclotomic main conjectures for elliptic curves},
    author={Bertolini, Massimo and Longo, Matteo and Venerucci, Rodolfo},
    year={2023},
    note={Preprint, arXiv:2306.17784},
   
}

@misc{BKNO,
    title={A local sign decomposition for symplectic self-dual Galois representations of rank two},
    author={Burungale, Ashay and Kobayashi, Shinichi and Nakamura, Kentaro and Ota, Kazuto},
    year={2025},
    note={Preprint, arXiv:2408.17776},
   
}

@article {CHKLL,
    AUTHOR = {Castella, Francesc and Hsu, Chi-Yun and Kundu, Debanjana and
              Lee, Yu-Sheng and Liu, Zheng},
     TITLE = {Derived {$p$}-adic heights and the leading coefficient of the
              {B}ertolini-{D}armon-{P}rasanna {$p$}-adic {$L$}-function},
   JOURNAL = {Trans. Amer. Math. Soc. Ser. B},
  FJOURNAL = {Transactions of the American Mathematical Society. Series B},
    VOLUME = {12},
      YEAR = {2025},
     PAGES = {748--788},
      
}

@article {burungale2020,
    AUTHOR = {Burungale, Ashay A.},
     TITLE = {On the non-triviality of the {$p$}-adic {A}bel-{J}acobi image
              of generalised {H}eegner cycles modulo {$p$}, {I}: {M}odular
              curves},
   JOURNAL = {J. Algebraic Geom.},
  FJOURNAL = {Journal of Algebraic Geometry},
    VOLUME = {29},
      YEAR = {2020},
    NUMBER = {2},
     PAGES = {329--371},
      ISSN = {1056-3911,1534-7486},
   MRCLASS = {11R23 (11F85 14G35)},
  MRNUMBER = {4069652},
MRREVIEWER = {Jeanine\ Van Order},
       DOI = {10.1090/jag/748},
       URL = {https://doi.org/10.1090/jag/748},
}

@article {howard06,
    AUTHOR = {Howard, Benjamin},
     TITLE = {Bipartite {E}uler systems},
   JOURNAL = {J. Reine Angew. Math.},
  FJOURNAL = {Journal f\"{u}r die Reine und Angewandte Mathematik. [Crelle's
              Journal]},
    VOLUME = {597},
      YEAR = {2006},
     PAGES = {1--25},
     }

@article{cst,
AUTHOR = {Cai, Li and Shu, Jie and Tian, Ye},
     TITLE = {Explicit {G}ross-{Z}agier and {W}aldspurger formulae},
   JOURNAL = {Algebra Number Theory},
  FJOURNAL = {Algebra \& Number Theory},
    VOLUME = {8},
      YEAR = {2014},
    NUMBER = {10},
     PAGES = {2523--2572},
      ISSN = {1937-0652,1944-7833},
   MRCLASS = {11G40},
  MRNUMBER = {3298547},
MRREVIEWER = {Steven\ Joel\ Miller},
       DOI = {10.2140/ant.2014.8.2523},
       URL = {https://doi.org/10.2140/ant.2014.8.2523},
}

@article {perrinriou87,
    AUTHOR = {Perrin-Riou, Bernadette},
     TITLE = {Points de {H}eegner et d\'{e}riv\'{e}es de fonctions {$L$}
              {$p$}-adiques},
   JOURNAL = {Invent. Math.},
  FJOURNAL = {Inventiones Mathematicae},
    VOLUME = {89},
      YEAR = {1987},
    NUMBER = {3},
     PAGES = {455--510},
      
}

@MISC{CCSS,
Author={Castella, Francesc and Ciperiani, Mirela and Skinner, Christopher and Sprung, Florian},
Title={The {I}wasawa main conjectures for modular forms at non-ordinary primes},
note={preprint,  arXiv:1804.10993},
year={2018},
}

@article {BD-mumford,
    AUTHOR = {Bertolini, Massimo. and Darmon, Henri},
     TITLE = {Heegner points on {M}umford-{T}ate curves},
   JOURNAL = {Invent. Math.},
  FJOURNAL = {Inventiones Mathematicae},
    VOLUME = {126},
      YEAR = {1996},
    NUMBER = {3},
     PAGES = {413--456},
     
}

@article {zhang14,
    AUTHOR = {Zhang, Wei},
     TITLE = {Selmer groups and the indivisibility of {H}eegner points},
   JOURNAL = {Camb. J. Math.},
  FJOURNAL = {Cambridge Journal of Mathematics},
    VOLUME = {2},
      YEAR = {2014},
    NUMBER = {2},
     PAGES = {191--253},
      
}

@article {jsw,
    AUTHOR = {Jetchev, Dimitar and Skinner, Christopher and Wan, Xin},
     TITLE = {The {B}irch and {S}winnerton-{D}yer formula for elliptic
              curves of analytic rank one},
   JOURNAL = {Camb. J. Math.},
  FJOURNAL = {Cambridge Journal of Mathematics},
    VOLUME = {5},
      YEAR = {2017},
    NUMBER = {3},
     PAGES = {369--434},
      ISSN = {2168-0930},
   MRCLASS = {11G40 (11G05 11G07 11R23)},
  
}

@article{BBL1,
    AUTHOR = {Burungale, Ashay and B\"{u}y\"{u}kboduk, K\^{a}zim and Lei, Antonio},
     TITLE = {Anticyclotomic {I}wasawa theory of abelian varieties of {$\mathrm{GL}_2$}-type at non-ordinary primes},
     journal = {Adv. Math.},
volume = {439},
pages = {109465},
year = {2024},
    }

@article {DD_parity,
    AUTHOR = {Dokchitser, Tim and Dokchitser, Vladimir},
     TITLE = {On the {B}irch-{S}winnerton-{D}yer quotients modulo squares},
   JOURNAL = {Ann. of Math. (2)},
  FJOURNAL = {Annals of Mathematics. Second Series},
    VOLUME = {172},
      YEAR = {2010},
    NUMBER = {1},
     PAGES = {567--596},
      ISSN = {0003-486X,1939-8980},
   MRCLASS = {11G40 (11G05)},
  MRNUMBER = {2680426},
MRREVIEWER = {Ramdorai\ Sujatha},
       DOI = {10.4007/annals.2010.172.567},
       URL = {https://doi.org/10.4007/annals.2010.172.567},
}

@article {BDKim_Parity,
    AUTHOR = {Kim, Byoung Du},
     TITLE = {The parity conjecture for elliptic curves at supersingular
              reduction primes},
   JOURNAL = {Compos. Math.},
  FJOURNAL = {Compos. Mathematica},
    VOLUME = {143},
      YEAR = {2007},
    NUMBER = {1},
     PAGES = {47--72},
      ISSN = {0010-437X,1570-5846},
   MRCLASS = {11G05 (11G40 11R23)},
  MRNUMBER = {2295194},
MRREVIEWER = {Benjamin\ V.\ Howard},
       DOI = {10.1112/S0010437X06002569},
       URL = {https://doi.org/10.1112/S0010437X06002569},
}

@article {AndreattaIovitaBDP,
    AUTHOR = {Andreatta, Fabrizio and Iovita, Adrian},
     TITLE = {{K}atz type $p$-adic ${L}$-functions for primes $p$ non-split in the {CM} field},
  year = {2024},
journal = {Comment. Math. Helv.},
number={4},
volume={99},
pages={641--716}
}

@article {BertoliniDarmon2005,
    AUTHOR = {Bertolini, Massimo and Darmon, Henri},
     TITLE = {Iwasawa's main conjecture for elliptic curves over
              anticyclotomic {$\Bbb Z_p$}-extensions},
   JOURNAL = {Ann. of Math. (2)},
  FJOURNAL = {Annals of Mathematics. Second Series},
    VOLUME = {162},
      YEAR = {2005},
    NUMBER = {1},
     PAGES = {1--64},
      
}

@unpublished{sweeting,
author={Sweeting, Naomi},
title={Kolyvagin's Conjecture and patched {E}uler systems in anticyclotomic {I}wasawa theory},
year={2024},
note={Preprint,  \href{https://scholar.harvard.edu/sites/scholar.harvard.edu/files/naomisweeting/files/kolyvagin.pdf}{ar{X}iv:2012.11771}}
}

@article {GZ86,
    AUTHOR = {Gross, Benedict H. and Zagier, Don B.},
     TITLE = {Heegner points and derivatives of {$L$}-series},
   JOURNAL = {Invent. Math.},
  FJOURNAL = {Inventiones Mathematicae},
    VOLUME = {84},
      YEAR = {1986},
    NUMBER = {2},
     PAGES = {225--320},
      
}

@article {cornut,
    AUTHOR = {Cornut, Christophe},
     TITLE = {Mazur's conjecture on higher {H}eegner points},
   JOURNAL = {Invent. Math.},
  FJOURNAL = {Inventiones Mathematicae},
    VOLUME = {148},
      YEAR = {2002},
    NUMBER = {3},
     PAGES = {495--523},
      ISSN = {0020-9910,1432-1297},
   MRCLASS = {11G40 (11F33 11F67 11G18 11R23)},
  MRNUMBER = {1908058},
MRREVIEWER = {Jan\ Nekov\'a\v r},
       DOI = {10.1007/s002220100199},
       URL = {https://doi.org/10.1007/s002220100199},
}

@Article{bertolinidarmonprasanna13,
    Author = {Massimo {Bertolini} and Henri {Darmon} and Kartik {Prasanna}},
    Title = {{Generalized Heegner cycles and $p$-adic Rankin $L$-series.}},
    FJournal = {{Duke Mathematical Journal}},
    Journal = {{Duke Math. J.}},
    ISSN = {0012-7094; 1547-7398/e},
    Volume = {162},
    Number = {6},
    Pages = {1033--1148},
    Year = {2013},
 
}

@article {castellawan,
    AUTHOR = {Castella, Francesc and Wan, Xin},
     TITLE = {Perrin-{R}iou's main conjecture for elliptic curves at
              supersingular primes},
   JOURNAL = {Math. Ann.},
  FJOURNAL = {Mathematische Annalen},
    VOLUME = {389},
      YEAR = {2024},
    NUMBER = {3},
     PAGES = {2595--2636},
      
}

@article {wan21,
    AUTHOR = {Wan, Xin},
     TITLE = {Heegner point {K}olyvagin system and {I}wasawa main
              conjecture},
   JOURNAL = {Acta Math. Sin. (Engl. Ser.)},
  FJOURNAL = {Acta Mathematica Sinica (English Series)},
    VOLUME = {37},
      YEAR = {2021},
    NUMBER = {1},
     PAGES = {104--120},
    }

@article {longovigni,
    AUTHOR = {Longo, Matteo and Vigni, Stefano},
     TITLE = {Plus/minus {H}eegner points and {I}wasawa theory of elliptic
              curves at supersingular primes},
   JOURNAL = {Boll. Unione Mat. Ital.},
  FJOURNAL = {Bollettino dell'Unione Matematica Italiana},
    VOLUME = {12},
      YEAR = {2019},
    NUMBER = {3},
     PAGES = {315--347},
     
}

@article {skinnerurbanmainconj,
    AUTHOR = {Skinner, Christopher and Urban, Eric},
     TITLE = {The {I}wasawa main conjectures for {$\rm GL_2$}},
   JOURNAL = {Invent. Math.},
  FJOURNAL = {Inventiones Mathematicae},
    VOLUME = {195},
      YEAR = {2014},
    NUMBER = {1},
     PAGES = {1--277},
    
}

@article {skinner20,
    AUTHOR = {Skinner, Christopher},
     TITLE = {A converse to a theorem of {G}ross, {Z}agier, and {K}olyvagin},
   JOURNAL = {Ann. of Math. (2)},
  FJOURNAL = {Annals of Mathematics. Second Series},
    VOLUME = {191},
      YEAR = {2020},
    NUMBER = {2},
     PAGES = {329--354},
      ISSN = {0003-486X,1939-8980},
   MRCLASS = {11G40 (11G05 11G07 11R23)},
  MRNUMBER = {4076627},
MRREVIEWER = {Andreas\ Nickel},
       DOI = {10.4007/annals.2020.191.2.1},
       URL = {https://doi.org/10.4007/annals.2020.191.2.1},
}

@incollection {bst2021,
    AUTHOR = {Burungale, Ashay A. and Skinner, Christopher and Tian, Ye},
     TITLE = {The {B}irch and {S}winnerton-{D}yer conjecture: a brief
              survey},
 BOOKTITLE = {Nine mathematical challenges---an elucidation},
    SERIES = {Proc. Sympos. Pure Math.},
    VOLUME = {104},
     PAGES = {11--29},
 PUBLISHER = {Amer. Math. Soc., Providence, RI},
      YEAR = {2021},
      ISBN = {978-1-4704-5490-6},
   MRCLASS = {11G40},
  MRNUMBER = {4337415},
       DOI = {10.1090/pspum/104/01876},
       URL = {https://doi.org/10.1090/pspum/104/01876},
}

@article {chidahsiehanticyclomainconjformodformscomposito,
    AUTHOR = {Chida, Masataka and Hsieh, Ming-Lun},
     TITLE = {On the anticyclotomic {I}wasawa main conjecture for modular
              forms},
   JOURNAL = {Compos. Math.},
  FJOURNAL = {Compos. Mathematica},
    VOLUME = {151},
      YEAR = {2015},
    NUMBER = {5},
     PAGES = {863--897},
     
}

@ARTICLE {kbb,
    AUTHOR = {B{\"{u}}y{\"{u}}kboduk, K\^{a}z{\i}m},
     TITLE = {{$\Lambda$}-adic {K}olyvagin systems},
     JOURNAL = {IMRN},
  FJOURNAL = {International Math Research Notices},
    VOLUME = {2011},
      YEAR = {2011},
       NUMBER = {14},
     PAGES = {3141--3206},
}

@article {howard04,
    AUTHOR = {Howard, Benjamin},
     TITLE = {The {H}eegner point {K}olyvagin system},
   JOURNAL = {Compos. Math.},
  FJOURNAL = {Compos. Mathematica},
    VOLUME = {140},
      YEAR = {2004},
    NUMBER = {6},
     PAGES = {1439--1472},
      ISSN = {0010-437X},
   MRCLASS = {11G05 (11R23)},
  MRNUMBER = {2098397},
MRREVIEWER = {Henri Darmon},
       DOI = {10.1112/S0010437X04000569},
       URL = {https://doi.org/10.1112/S0010437X04000569},
}

@article {Helm07,
    AUTHOR = {Helm, David},
     TITLE = {On maps between modular {J}acobians and {J}acobians of
              {S}himura curves},
   JOURNAL = {Israel J. Math.},
  FJOURNAL = {Israel Journal of Mathematics},
    VOLUME = {160},
      YEAR = {2007},
     PAGES = {61--117},
      ISSN = {0021-2172},
   MRCLASS = {11G18 (14G35)},
  MRNUMBER = {2342491},
MRREVIEWER = {Jacques Tilouine},
       DOI = {10.1007/s11856-007-0056-0},
       URL = {https://doi.org/10.1007/s11856-007-0056-0},
}

@article {Buzzard97,
    AUTHOR = {Buzzard, Kevin},
     TITLE = {Integral models of certain {S}himura curves},
   JOURNAL = {Duke Math. J.},
  FJOURNAL = {Duke Mathematical Journal},
    VOLUME = {87},
      YEAR = {1997},
    NUMBER = {3},
     PAGES = {591--612},
      ISSN = {0012-7094},
   MRCLASS = {11G18 (11G05 14G25 14K15)},
  MRNUMBER = {1446619},
MRREVIEWER = {Conjeeveram S. Rajan},
       DOI = {10.1215/S0012-7094-97-08719-6},
       URL = {https://doi.org/10.1215/S0012-7094-97-08719-6},
}

@book {KatzMazur85,
    AUTHOR = {Katz, Nicholas M. and Mazur, Barry},
     TITLE = {Arithmetic moduli of elliptic curves},
    SERIES = {Annals of Mathematics Studies},
    VOLUME = {108},
 PUBLISHER = {Princeton University Press, Princeton, NJ},
      YEAR = {1985},
     PAGES = {xiv+514},
      ISBN = {0-691-08349-5; 0-691-08352-5},
   MRCLASS = {11G05 (11F11 14G25 14K15)},
  MRNUMBER = {772569},
MRREVIEWER = {Kenneth A. Ribet},
       DOI = {10.1515/9781400881710},
       URL = {https://doi.org/10.1515/9781400881710},
}

@article {chidahseihanticycloLvaluescrelle,
    AUTHOR = {Chida, Masataka and Hsieh, Ming-Lun},
     TITLE = {Special values of anticyclotomic {$L$}-functions for modular
              forms},
   JOURNAL = {J. Reine Angew. Math.},
  FJOURNAL = {Journal f\"{u}r die Reine und Angewandte Mathematik. [Crelle's
              Journal]},
    VOLUME = {741},
      YEAR = {2018},
     PAGES = {87--131},
}

@article {burungale2017,
    AUTHOR = {Burungale, Ashay A.},
     TITLE = {On the non-triviality of the {$p$}-adic {A}bel-{J}acobi image
              of generalised {H}eegner cycles modulo {$p$}, {II}: {S}himura
              curves},
   JOURNAL = {J. Inst. Math. Jussieu},
  FJOURNAL = {Journal of the Institute of Mathematics of Jussieu. JIMJ.
              Journal de l'Institut de Math\'{e}matiques de Jussieu},
    VOLUME = {16},
      YEAR = {2017},
    NUMBER = {1},
     PAGES = {189--222},
  
}

@article {FKP,
    AUTHOR = {Fakhruddin, Najmuddin and Khare, Chandrashekhar and Patrikis,
              Stefan},
     TITLE = {Relative deformation theory, relative {S}elmer groups, and
              lifting irreducible {G}alois representations},
   JOURNAL = {Duke Math. J.},
  FJOURNAL = {Duke Mathematical Journal},
    VOLUME = {170},
      YEAR = {2021},
    NUMBER = {16},
     PAGES = {3505--3599},
  
}

@article {Bertolini1995,
    AUTHOR = {Bertolini, Massimo},
     TITLE = {Selmer groups and {H}eegner points in anticyclotomic {$\bold
              Z_p$}-extensions},
   JOURNAL = {Compos. Math.},
  FJOURNAL = {Compos. Mathematica},
    VOLUME = {99},
      YEAR = {1995},
    NUMBER = {2},
     PAGES = {153--182},
    
}

@article {BCK,
    AUTHOR = {Burungale, Ashay and Castella, Francesc and Kim, Chan-Ho},
     TITLE = {A proof of {P}errin-{R}iou's {H}eegner point main conjecture},
   JOURNAL = {Algebra Number Theory},
  FJOURNAL = {Algebra \& Number Theory},
    VOLUME = {15},
      YEAR = {2021},
    NUMBER = {7},
     PAGES = {1627--1653},
      
}

@article {koly,
    AUTHOR = {Kolyvagin, V. A.},
     TITLE = {Finiteness of {$E({\bf Q})$} and {$CH(E, {\bf Q})$} for a
              subclass of {W}eil curves},
   JOURNAL = {Izv. Akad. Nauk SSSR Ser. Mat.},
  FJOURNAL = {Izvestiya Akademii Nauk SSSR. Seriya Matematicheskaya},
    VOLUME = {52},
      YEAR = {1988},
    NUMBER = {3},
     PAGES = {522--540, 670--671},
      ISSN = {0373-2436},
   MRCLASS = {11G05 (11G40 14G25 14K07)},
  MRNUMBER = {954295},
MRREVIEWER = {Reinhard\ B\"{o}lling},
       DOI = {10.1070/IM1989v032n03ABEH000779},
       URL = {https://doi.org/10.1070/IM1989v032n03ABEH000779},
}

@article {kbbdeform,
    AUTHOR = {B{\"u}y{\"u}kboduk, K\^{a}z{\i}m},
     TITLE = {Deformations of {K}olyvagin systems},
   JOURNAL = {Ann. of Math. Qu\'{e}.},
  FJOURNAL = {Annales Math\'{e}matiques du Qu\'{e}bec},
    VOLUME = {40},
      YEAR = {2016},
    NUMBER = {2},
     PAGES = {251--302},
      ISSN = {2195-4755,2195-4763},
   MRCLASS = {11R23 (11G05 11G10 11G40 14G10)},
  MRNUMBER = {3529183},
MRREVIEWER = {Soogil\ Seo},
       DOI = {10.1007/s40316-015-0044-4},
       URL = {https://doi.org/10.1007/s40316-015-0044-4},
}

@article {mr02,
    AUTHOR = {Mazur, Barry and Rubin, Karl},
     TITLE = {Kolyvagin systems},
   JOURNAL = {Mem. Amer. Math. Soc.},
  FJOURNAL = {Memoirs of the American Mathematical Society},
    VOLUME = {168},
      YEAR = {2004},
    NUMBER = {799},
     PAGES = {viii+96},
    
}

@INCOLLECTION{blochkato,
author={Bloch, Spencer and Kato, Kazuya},
title={{$L$}-functions and {T}amagawa numbers of motives},
booktitle={The {G}rothendieck {F}estschrift, {V}ol.\ {I}},
pages={333--400}, 
series={Progr. Math.}, 
volume={86}, 
publisher={Birkh\"auser Boston}, 
address={Boston, MA},
year={1990},}

@article {darmoniovita,
    AUTHOR = {Darmon, Henri and Iovita, Adrian},
     TITLE = {The anticyclotomic main conjecture for elliptic curves at
              supersingular primes},
   JOURNAL = {J. Inst. Math. Jussieu},
  FJOURNAL = {Journal of the Institute of Mathematics of Jussieu. JIMJ.
              Journal de l'Institut de Math\'ematiques de Jussieu},
    VOLUME = {7},
      YEAR = {2008},
    NUMBER = {2},
     PAGES = {291--325},
     
}

@ARTICLE{kobayashi03,
  author = {Kobayashi, Shinichi},
  title = {Iwasawa theory for elliptic curves at supersingular primes},
  journal = {Invent. Math.},
  year = {2003},
  volume = {152},
  pages = {1--36},
  number = {1},
  fjournal = {Inventiones Mathematicae}
}

@article{pollack-weston11,
author={Pollack, Robert and Weston, Tom},
title={On anticyclotomic $\mu$-invariants of modular forms},
journal={Compos. Math.},
year={2011}, 
volume={147},
pages={1353--1381},
number={5},
fjournal={Compos. Mathematica}
}

\end{document}